\documentclass[a4paper]{amsart}

\usepackage[
    a4paper,
    top=30mm,
    bottom=30mm,
    left=25mm,
    right=25mm,
]{geometry}
\usepackage{amsmath}
\usepackage{amssymb}

\usepackage{ascmac}
\usepackage{amsfonts}
\usepackage{amsthm}
\usepackage{mathrsfs}
\usepackage{enumerate}
\usepackage{enumitem}
\usepackage{url}
\usepackage{tikz}
\usepackage[all]{xy}
\usepackage{graphicx}
\usepackage{xcolor}
\usepackage[normalem]{ulem}
\usepackage{hyperref}
\usepackage{cleveref}

\usepackage{fancyhdr}
\allowdisplaybreaks[4]

\theoremstyle{definition}
\newtheorem{main}{Main}[subsection]
\newtheorem{definition}[main]{Definition}
\newtheorem*{definition*}{Definition}
\newtheorem{thm}[main]{Theorem}
\newtheorem*{thm*}{Theorem}
\newtheorem{prop}[main]{Proposition}
\newtheorem*{prop*}{Proposition}
\newtheorem{lem}[main]{Lemma}
\newtheorem*{lem*}{Lemma}
\newtheorem{cor}[main]{Corollary}
\newtheorem*{cor*}{Corollary}
\newtheorem{rmk}[main]{Remark}
\newtheorem*{rmk*}{Remark}
\newtheorem{ex}[main]{Example}
\newtheorem*{ex*}{Example}

\newtheorem*{setting*}{Setting}
\newtheorem{problem}[main]{Problem}
\newtheorem*{problem*}{Problem}
\newtheorem{conjecture}[main]{Conjecture}
\newtheorem*{conjecture*}{Conjecture}

\Crefname{definition}{Definition}{Definitions}
\Crefname{thm}{Theorem}{Theorems}
\Crefname{prop}{Proposition}{Propositions}
\Crefname{lem}{Lemma}{Lemmas}
\Crefname{cor}{Corollary}{Corollaries}
\Crefname{rmk}{Remark}{Remarks}
\Crefname{ex}{Example}{Examples}
\Crefname{setting}{Setting}{Settings}
\Crefname{problem}{Problem}{Problems}
\Crefname{conjecture}{Conjecture}{Conjectures}

\newtheorem{mthm}{Main Theorem}

\Crefname{mthm}{Main Theorem}{Main Theorems}

\numberwithin{equation}{section}

\newcommand{\paren}[1]{\left( #1 \right)}
\newcommand{\bracket}[1]{\left[ #1 \right]}
\newcommand{\set}[1]{\left\{ #1 \right\}}
\newcommand{\abs}[1]{\left| #1 \right|}

\newcommand{\diam}{\mathrm{diam}}

\newcommand{\const}{\mathrm{const}}
\newcommand{\dimh}[1]{\dim_{H}\paren{#1}}
\newcommand{\dimp}[1]{\dim_{P}\paren{#1}}
\newcommand{\dimub}[1]{\overline{\dim}_{B}\paren{#1}}
\newcommand{\dimlb}[1]{\underline{\dim}_{B}\paren{#1}}
\newcommand{\iu}{\sqrt{-1}}
\newcommand{\Lip}{\mathrm{Lip}}
\newcommand{\Int}{\mathrm{Int}}
\newcommand{\Con}{\mathrm{Con}}
\newcommand{\Leb}{\mathrm{Leb}}

\DeclareMathOperator*{\argmax}{arg\,max}

\makeatletter
  \@addtoreset{equation}{subsection}
  
\makeatother

\makeatletter
\def\@settitle{%
  \begin{center}%
    \normalfont\Large\bfseries
    \@title\par
  \end{center}%
}
\makeatother

\title{The Limit Sets of Linear and Nonlinear Infinite IFSs Related to Complex Continued Fractions}
\author{Takumi Okamoto}

\address{Division of Mathematical and Information Sciences \\ Graduate School of Human and Environmental Studies \\
Kyoto University \\ 
Yoshida-nihonmatsu-cho, Sakyo-ku, Kyoto, 606-8501, Japan
}
\email{okamototakumi227@gmail.com}

\begin{document}

\begin{abstract}
    We introduce two families of infinite iterated function systems (IFSs) $\mathcal{F}(\mathbf{d}, T)$ and $\mathcal{G}(\mathbf{d}, T)$,
    parametrized by a sequence of positive real numbers $\mathbf{d}$ and a natural number $T$,
    and investigate the measure-theoretic properties of their limit sets.
    $\mathcal{G}(\mathbf{d}, T)$ is an infinite M\"{o}bius IFS, which is an extension of the IFS of real continued fractions to the IFS on the closed unit disc in the complex plane.
    $\mathcal{F}(\mathbf{d}, T)$ is an infinite linear IFS that shares the same first approximation as $\mathcal{G}(\mathbf{d}, T)$.
    We show that for many choices of $\mathbf{d}$ and $T$,
    the limit sets of both $\mathcal{F}(\mathbf{d}, T)$ and $\mathcal{G}(\mathbf{d}, T)$ exhibit phenomena unique to infinite IFSs,
    such as having zero Hausdorff measure at the Hausdorff dimension, having infinite packing measure at the packing dimension,
    or having different Hausdorff and packing dimensions.
    We also prove that the Hausdorff dimension of the limit set of $\mathcal{F}(\mathbf{d}, T)$ is strictly larger than that of $\mathcal{G}(\mathbf{d}, T)$
    under certain conditions on the parameters.
\end{abstract}


\maketitle

\section*{Notation}

In this paper, we use the following notation.

\begin{center}
    \begin{tabular}{|c|l|}
        \hline
        Symbol & Meaning \\
        \hline
        $\mathbb{N}$ & The set of positive integers \\
        $\mathbb{Z}$ & The set of integers \\
        $\mathbb{R}$ & The set of real numbers \\
        $\mathbb{C}$ & The set of complex numbers \\
        $\mathbb{D}$ & The unit disc $\set{z \in \mathbb{C} \mid \abs{z} < 1}$ \\
        $B(x, r)$ & The open ball centered at $x$ with radius $r$ $\set{y \in \mathbb{R}^{d} \mid \abs{y-x} < r}$ \\
        $\overline{B}(x, r)$ & The closed ball centered at $x$ with radius $r$ $\set{y \in \mathbb{R}^{d} \mid \abs{y-x} \leq r}$ \\
        $\diam(A)$ & The diameter of the set $A$ \\
        $\Int(A)$ & The interior of the set $A$ \\
        $\Leb_{d}$ & The $d$-dimensional Lebesgue measure on $\mathbb{R}^{d}$ \\
        $\dimh{A}$ & The Hausdorff dimension of the set $A$ \\
        $\dimp{A}$ & The packing dimension of the set $A$ \\
        $\dimub{A}$ & The upper box dimension of the set $A$ \\
        $\dimlb{A}$ & The lower box dimension of the set $A$ \\
        $H_{t}(A)$ & The $t$-dimensional Hausdorff measure of the set $A$ \\
        $\Pi_{t}(A)$ & The $t$-dimensional packing measure of the set $A$ \\
        $\iu$ & The imaginary unit \\
        $e(\theta)$ & $e^{2\pi\iu \theta}$ \\
        $\Lip(s)$ & The Lipschitz constant of the contraction map $s$ \\
        $\Con(x, u, \alpha, l)$ & \parbox[t]{0.7\textwidth}{The open cone with vertex $x$, direction vector $u$, central angle $\alpha$, and altitude $l$ \\ $\set{y \in \mathbb{R}^{d} \mid 0 < \abs{y-x} < l, \frac{\langle y-x, u \rangle_{\mathbb{R}^{d}}}{\abs{y-x}\abs{u}} > \cos\alpha }$} \\
        $\lfloor a \rfloor$ & The greatest integer less than or equal to the real number $a$ \\
        $\lceil a \rceil$ & The smallest integer greater than or equal to the real number $a$ \\
        $\mathrm{Re}(z)$ & The real part of the complex number $z$ \\
        $\mathrm{Im}(z)$ & The imaginary part of the complex number $z$ \\
        \hline
    \end{tabular}
\end{center}

\section{Introduction}\label{sec:introduction}

\subsection{Historical Background}
Typical fractal sets, such as the middle-third Cantor set and the Sierpi\'{n}ski gasket,
can be realized as the limit sets of iterated function systems (IFSs) consisting of finitely many contracting similarities.
The theory of IFSs consisting of finitely many contracting similarities has been extensively studied.
Let $\mathcal{S}$ be an IFS of this type, and let $J_{\mathcal{S}}$ be the limit set of $\mathcal{S}$.
Hutchinson proved the existence of a unique self-similar measure for $J_{\mathcal{S}}$
and showed that the Hausdorff dimension of $J_{\mathcal{S}}$ is equal to the similarity dimension of $\mathcal{S}$  under the open set condition (\cite{Hutchinson1981}).
Furthermore, for such $\mathcal{S}$,
it is known that the Hausdorff, packing, and box dimensions of $J_{\mathcal{S}}$ coincide,
and that $J_{\mathcal{S}}$ is an $s$-set (\cite{Falconer}, \cite{yamaguchi-hata-kigami-en}).

It is a natural and important question whether analogous results remain valid
when the assumptions on $\mathcal{S}$ are relaxed.
Mauldin and Williams proved that similar results hold
even when $\mathcal{S}$ is an IFS consisting of infinitely many similitudes with the open set condition (\cite{MauldinWilliams1986}).
Mauldin and Urba\'{n}ski studied the dimension and measure of the limit sets of IFSs
satisfying the open set condition and consisting of countably many conformal contractions (\cite{MauldinUrbanski1996}, \cite{MauldinUrbanski1999}).
In the case where $\mathcal{S}$ is such an infinite IFS,
it has been proved that the following three phenomena, which do not occur for finite IFSs, can occur (\cite{Mauldin1995}, \cite{MauldinUrbanski1996}, \cite{MauldinUrbanski1999}).

\begin{itemize}
    \item[\textbf{(C1)}] $H_{\dimh{J_{\mathcal{S}}}}(J_{\mathcal{S}}) = 0$.
    \item[\textbf{(C2)}] $\Pi_{\dimp{J_{\mathcal{S}}}}(J_{\mathcal{S}}) = \infty$.
    \item[\textbf{(C3)}] $\dimh{J_{\mathcal{S}}} < \dimlb{J_{\mathcal{S}}} \leq \dimub{J_{\mathcal{S}}} = \dimp{J_{\mathcal{S}}}$.
\end{itemize}

In particular, they showed the following results on real continued fractions.
For each infinite subset $I$ of $\mathbb{N}$, let $\widetilde{\mathcal{G}}(I)$ be the IFS on $[0, 1]$
generated by the maps $\set{x \mapsto 1/(x+n)}_{n \in I}$.

\begin{thm}[\cite{MauldinUrbanski1999} Theorem 6.1] \label{thm:MauldinUrbanski1999-6.1}
    For $p \in \mathbb{N}$ with $p \geq 2$, let $I_{p} := \set{n^{p} \mid n \in \mathbb{N}}$.
    Define $\widetilde{\mathcal{G}}^{(p)} := \widetilde{\mathcal{G}}(I_{p})$.
    Then, the following holds.
    \begin{equation*}
        \dimh{J_{\widetilde{\mathcal{G}}^{(p)}}} = \dimp{J_{\widetilde{\mathcal{G}}^{(p)}}} > 1/p, \quad 0 < H_{\dimh{J_{\widetilde{\mathcal{G}}^{(p)}}}}\paren{J_{\widetilde{\mathcal{G}}^{(p)}}} < \infty, \quad \Pi_{\dimh{J_{\widetilde{\mathcal{G}}^{(p)}}}}\paren{J_{\widetilde{\mathcal{G}}^{(p)}}} = \infty
    \end{equation*}
\end{thm}

\begin{thm}[\cite{MauldinUrbanski1999} Theorem 6.2] \label{thm:MauldinUrbanski1999-6.2}
    For $p \in \mathbb{N}$ with $p \geq 2$ and $l \in \mathbb{N}$,
    define $I_{p, l} := \set{n^{p} \mid n \in \mathbb{N}, n \geq l}$.
    Let $\widetilde{\mathcal{G}}^{(p, l)} := \widetilde{\mathcal{G}}(I_{p, l})$.
    Then, for each $p \in \mathbb{N}$ with $p \geq 2$, there exists $l_{p} \in \mathbb{N}$ such that
    for all $l \in \mathbb{N}$ with $l \geq l_{p}$,
    \begin{equation*}
        \dimh{J_{\widetilde{\mathcal{G}}^{(p, l)}}} < \dimlb{J_{\widetilde{\mathcal{G}}^{(p, l)}}} \leq \dimub{J_{\widetilde{\mathcal{G}}^{(p, l)}}} = \dimp{J_{\widetilde{\mathcal{G}}^{(p, l)}}}
    \end{equation*}
    holds.
\end{thm}

The proof of \Cref{thm:MauldinUrbanski1999-6.1} relies on the following result.
\begin{thm}[\cite{MauldinUrbanski1999} Proposition 4.4, Lemma 5.2]\label{thm:MauldinUrbanski1999-4.4-5.2}
    Let $I$ be an infinite subset of $\mathbb{N}$. Define $\theta_{\widetilde{\mathcal{G}}(I)}$ by
    \begin{equation*}
        \theta_{\widetilde{\mathcal{G}}(I)} := \inf\set{t \geq 0 \mid P_{\widetilde{\mathcal{G}}(I)}(t) < \infty},
    \end{equation*}
    where $P_{\widetilde{\mathcal{G}}(I)}$ is the pressure function of $\widetilde{\mathcal{G}}(I)$ (see \Cref{def:pressure_function} for the pressure function).
    Then, the following statements hold.
    \begin{enumerate}
        \item If $\dimh{J_{\widetilde{\mathcal{G}}(I)}} < 2\theta_{\widetilde{\mathcal{G}}(I)}$, then $H_{\dimh{J_{\widetilde{\mathcal{G}}(I)}}}\paren{J_{\widetilde{\mathcal{G}}(I)}} = 0$.
        \item If $\dimh{J_{\widetilde{\mathcal{G}}(I)}} > 2\theta_{\widetilde{\mathcal{G}}(I)}$, then $\Pi_{\dimh{J_{\widetilde{\mathcal{G}}(I)}}}\paren{J_{\widetilde{\mathcal{G}}(I)}} = \infty$.
    \end{enumerate}
\end{thm}

In this paper,
we extend these nonlinear one-dimensional infinite IFSs to the infinite IFSs on the closed unit disc in the complex plane
and show that some of the above results generalize to two-dimensional IFSs.
In addition, we introduce an infinite IFS consisting of similarities
that corresponds, in a certain sense, to the nonlinear system above.
These two infinite IFSs provide many examples of the limit sets
exhibiting phenomena specific to infinite IFSs.
We also compare the Hausdorff dimensions of the limit sets of these two IFSs.

\subsection{Main Results}
\begin{definition}[\Cref{def:f-and-g}]
    For each complex number $\alpha$ with $\abs{\alpha} \geq 2$,
    we define $f_{\alpha}, g_{\alpha}: \overline{\mathbb{D}} \to \overline{\mathbb{D}}$ by
    \begin{equation*}
        f_{\alpha}(z) := \frac{z + \overline{\alpha}}{\abs{\alpha}^{2}-1}, \quad g_{\alpha}(z) := \frac{1}{z + \alpha}  \quad (z \in \overline{\mathbb{D}}).
    \end{equation*}
\end{definition}

\begin{definition}[\Cref{def:d}]
    We define the set of sequences of real numbers $\mathcal{D}$ by
    \begin{equation*}
        \mathcal{D} := \set{\set{d_n}_{n=1}^{\infty} \mid d_1 \geq 2 \text{ and } d_{n+1} - d_n \geq 2 \text{ holds for all } n \in \mathbb{N}}.
    \end{equation*}
\end{definition}

\begin{definition}[\Cref{def:f-and-g-ifs}]
    Let $\mathbf{d} = \set{d_n}_{n=1}^{\infty} \in \mathcal{D}$ and $T \in \mathbb{N}$.
    We define the IFSs $\mathcal{F} = \mathcal{F}\paren{\mathbf{d}, T}$ and
    $\mathcal{G} = \mathcal{G}\paren{\mathbf{d}, T}$ on $\overline{\mathbb{D}}$ as follows.
    \begin{align*}
        \mathcal{F} &:= \set{f_{e(j/T)d_{n}} \mid n \in\mathbb{N},\ j=0,1, \dots, T-1}, \\
        \mathcal{G} &:= \set{g_{e(j/T)d_{n}} \mid n \in\mathbb{N},\ j=0,1, \dots, T-1}.
    \end{align*}
\end{definition}

\begin{figure}
    \centering
    \begin{minipage}{0.48\textwidth}
        \centering
        \includegraphics[width=0.9\textwidth]{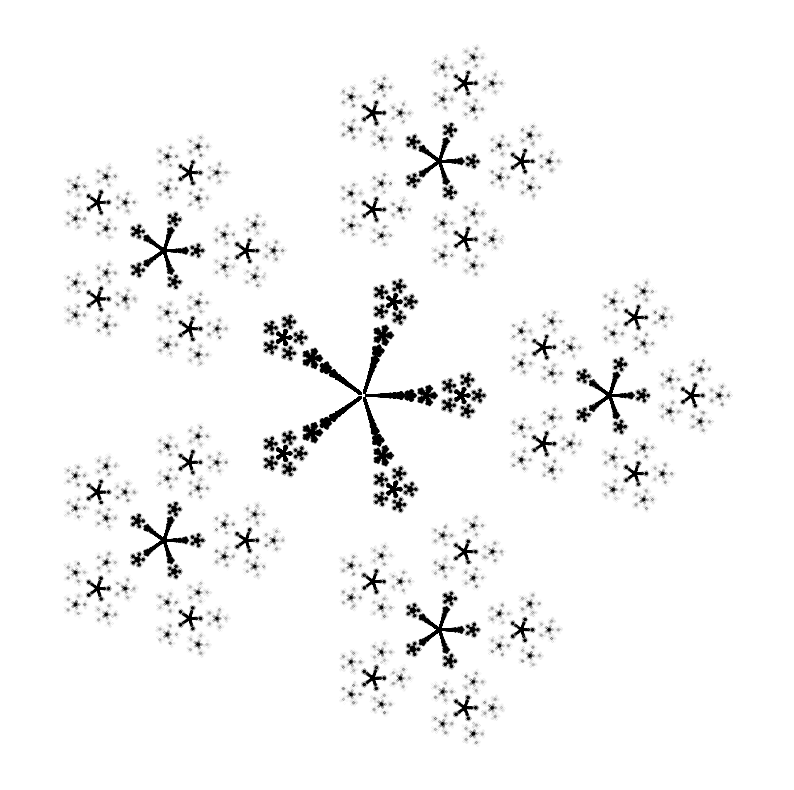}
    \end{minipage}
    \hfill
    \begin{minipage}{0.48\textwidth}
        \centering
        \includegraphics[width=0.9\textwidth]{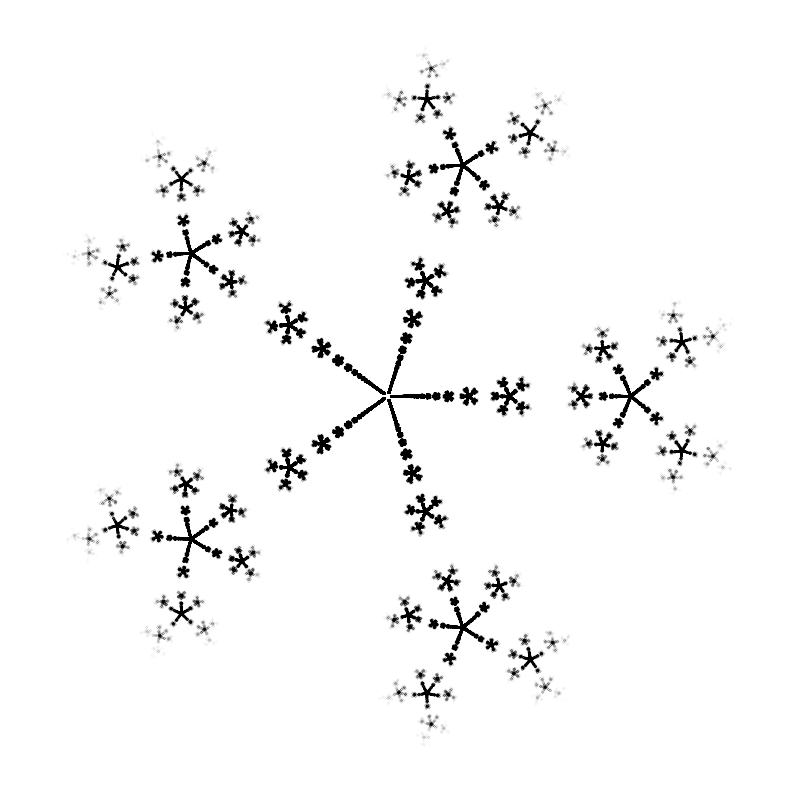}
    \end{minipage}

    \caption{Examples of the limit sets. The left figure shows $J_{\mathcal{F}(\mathbf{d}, T)}$, and the right figure shows $J_{\mathcal{G}(\mathbf{d}, T)}$.
    The parameters are $d_{n} = 2n + 10^{-9}$ and $T=5$.}
\end{figure}

If $T$ and $d_{1}$ satisfy suitable conditions,
then both $\mathcal{F}$ and $\mathcal{G}$ are conformal iterated function systems (CIFSs) (see \Cref{lem:cifs-sufficient-cond} and \Cref{def:CIFS}).
When $\mathcal{F}$ is a CIFS, let $J_{\mathcal{F}}$ and $P_{\mathcal{F}}$ be the limit set and the pressure function of $\mathcal{F}$, respectively.
We define $J_{\mathcal{G}}$ and $P_{\mathcal{G}}$ similarly for $\mathcal{G}$.

The following theorem is based on Lemma 4.9 and a revised version of Lemma 4.13 in \cite{MauldinUrbanski1996}, which relate the density to the dimensions.
This theorem also parallels \Cref{thm:MauldinUrbanski1999-4.4-5.2}.

\begin{mthm}[\Cref{thm:density-at-zero-and-packing-measure-infinity-main1}]\label{main1}
  Let $\mathbf{d} \in \mathcal{D}$ and $T \in \mathbb{N}$. Then the following statements hold.
  \begin{enumerate}
    \item Suppose that $\mathcal{F} = \mathcal{F}\paren{\mathbf{d}, T}$ is a regular CIFS
      (see \Cref{def:regularity} for the definition of a regular CIFS).
      If $P_{\mathcal{F}}(h_{\mathcal{F}}/2) < \infty$, then we have
      \begin{equation*}
        \limsup_{r\downarrow 0}\frac{m_{\mathcal{F}}\paren{B(0, r)}}{r^{h_{\mathcal{F}}}} = 0, \quad \Pi_{h_{\mathcal{F}}}\paren{J_{\mathcal{F}}} = \infty,
      \end{equation*}
      where $h_{\mathcal{F}} = \dimh{J_{\mathcal{F}}}$ and $m_{\mathcal{F}}$ is the $h_{\mathcal{F}}$-conformal measure for $\mathcal{F}$ (see \Cref{def:conformal-measure} for its definition).
    \item Suppose that $\mathcal{G} = \mathcal{G}\paren{\mathbf{d}, T}$ is a regular CIFS.
      If $P_{\mathcal{G}}(h_{\mathcal{G}}/2) < \infty$, then we have
      \begin{equation*}
        \limsup_{r\downarrow 0}\frac{m_{\mathcal{G}}\paren{B(0, r)}}{r^{h_{\mathcal{G}}}} = 0, \quad \Pi_{h_{\mathcal{G}}}\paren{J_{\mathcal{G}}} = \infty,
      \end{equation*}
      where $h_{\mathcal{G}} = \dimh{J_{\mathcal{G}}}$ and $m_{\mathcal{G}}$ is the $h_{\mathcal{G}}$-conformal measure for $\mathcal{G}$.
  \end{enumerate}
\end{mthm}

For sequences $\mathbf{d} = \set{d_n}_{n=1}^{\infty} \in \mathcal{D}$ exhibiting polynomial growth with respect to $n$,
\Cref{main1} yields the following theorem,
which is analogous to \Cref{thm:MauldinUrbanski1999-6.1}.
This theorem establishes the existence of an abundance of the limit sets satisfying
the conditions \textbf{(C1)} or \textbf{(C2)} that are unique to infinite CIFSs.

\begin{mthm}[\Cref{thm:measure_poly_f}, \Cref{thm:measure_poly_g}, \Cref{cor:hausdorff-dim-equals-packing-dim-and-packing-measure-infty},  \Cref{thm:measure_poly_gamma_general_T_large}]\label{main2}
  Let $\mathbf{d} = \set{d_n}_{n=1}^{\infty} \in \mathcal{D}$ and $T \in \mathbb{N}$.
  Suppose that there exist constants $0 < c_{1} \leq c_{2} < \infty, \ c_{3} > 0, \ \gamma \geq 1$ and $N\in\mathbb{N}$
  such that for all $n \geq N$, $c_{1}n^{\gamma} \leq d_{n} \leq c_{2}n^{\gamma}$ and $d_{n+1} - d_{n} \geq c_{3}n^{\gamma-1}$ hold.
  We also assume that $c_{1}N^{\gamma} > 1$.
  Then the following (1) and (2) hold.
  \begin{enumerate}
    \item Suppose that $\mathcal{F} = \mathcal{F}\paren{\mathbf{d}, T}$ is a CIFS.
      Then, we have the following (a) to (c).
      \begin{enumerate}
        \item If $h_{\mathcal{F}} = \dimh{J_{\mathcal{F}}} < 1/\gamma$, then $H_{h_{\mathcal{F}}}\paren{J_{\mathcal{F}}} = 0$.
        \item If $h_{\mathcal{F}} = \dimh{J_{\mathcal{F}}} > 1/\gamma$, then $\dimp{J_{\mathcal{F}}} = \dimh{J_{\mathcal{F}}}$ and $\Pi_{h_{\mathcal{F}}}\paren{J_{\mathcal{F}}} = \infty$.
        \item If $T \geq c_{2}^{2/\gamma}(N+1)$, then $\dimh{J_{\mathcal{F}}} > 1/\gamma$.
      \end{enumerate}
    \item Suppose that $\mathcal{G} = \mathcal{G}\paren{\mathbf{d}, T}$ is a CIFS.
      Then, we have the following (a) to (c).
      \begin{enumerate}
        \item If $h_{\mathcal{G}} = \dimh{J_{\mathcal{G}}} < 1/\gamma$, then $H_{h_{\mathcal{G}}}\paren{J_{\mathcal{G}}} = 0$.
        \item If $h_{\mathcal{G}} = \dimh{J_{\mathcal{G}}} > 1/\gamma$, then $\dimp{J_{\mathcal{G}}} = \dimh{J_{\mathcal{G}}}$ and $\Pi_{h_{\mathcal{G}}}\paren{J_{\mathcal{G}}} = \infty$.
        \item If $T \geq c_{2}^{2/\gamma}(N+1)$, then $\dimh{J_{\mathcal{G}}} > 1/\gamma$.
      \end{enumerate}
    \end{enumerate}
\end{mthm}

Moreover, we have the following theorem, which is analogous to \Cref{thm:MauldinUrbanski1999-6.2}.
This theorem implies that a wide variety of the limit sets satisfy the condition \textbf{(C3)}, a phenomenon unique to infinite CIFSs.

\begin{mthm}[\Cref{thm:dimension_gap_poly}]\label{main3}
  Under the assumptions in \Cref{main2}, we also suppose that $\gamma > 1$
  and that $\mathcal{F}\paren{\mathbf{d}, T}$ and $\mathcal{G}\paren{\mathbf{d}, T}$ are CIFSs.
  For $l \in \mathbb{N}$, let
  \begin{equation*}
    \mathcal{F}^{(l)} := \mathcal{F}(\set{d_{n}}_{n=l}^{\infty}, T), \quad \mathcal{G}^{(l)} := \mathcal{G}(\set{d_{n}}_{n=l}^{\infty}, T).
  \end{equation*}
  Then, there exists $q = q(\mathbf{d}, T) \in \mathbb{N}$ such that for all $l\in\mathbb{N}$ with $l \geq q$,
  we have the following inequalities.
  \begin{align*}
    \dimh{J_{\mathcal{F}^{(l)}}} &< \dimlb{J_{\mathcal{F}^{(l)}}}\leq \dimub{J_{\mathcal{F}^{(l)}}} = \dimp{J_{\mathcal{F}^{(l)}}}, \\
    \dimh{J_{\mathcal{G}^{(l)}}} &< \dimlb{J_{\mathcal{G}^{(l)}}} \leq \dimub{J_{\mathcal{G}^{(l)}}} = \dimp{J_{\mathcal{G}^{(l)}}}.
  \end{align*}
\end{mthm}

For the IFS $\mathcal{S} = \set{s_{i}: X \to X}_{i \in I}$,
let $J_{\mathcal{S}}^{(N)} := \bigcup_{i_{1}, \dots, i_{N} \in I} s_{i_{1}} \circ \dots \circ s_{i_{N}}\paren{X}$.
If $\mathcal{S}$ is a CIFS, then there are some properties of $J_{\mathcal{S}}$
which depend only on $J_{\mathcal{S}}^{(1)}$
(see \Cref{prop:MauldinUrbanski1996-4.4} and \Cref{thm:MauldinUrbanski1996-4.5}).
We remark that $J_{\mathcal{F}}^{(1)} = J_{\mathcal{G}}^{(1)}$ (\Cref{prop:comparison_of_f_and_g}).
Thus, it may not be surprising that the statements in \Cref{main1,main2,main3} are common to both $\mathcal{F}$ and $\mathcal{G}$.
On the other hand, it is natural to ask whether $\mathcal{F}$ and $\mathcal{G}$ exhibit different properties.
In particular, it is interesting to compare the Hausdorff dimensions of $J_{\mathcal{F}}$ and $J_{\mathcal{G}}$.
The following theorems provide some answers to this question.

\begin{mthm}[\Cref{thm:comparison-of-dimension-f-and-g-irregular}]\label{main4}
    Let $\mathbf{d} \in \mathcal{D}$ and $T\in\mathbb{N}$.
    Suppose that
    $\mathcal{F} = \mathcal{F}(\mathbf{d}, T)$ and $\mathcal{G} = \mathcal{G}(\mathbf{d}, T)$ are CIFSs.
    Then, the following (1) and (2) hold.
    \begin{enumerate}
        \item If $\mathcal{F}$ is irregular (see \Cref{def:regularity}), then $\dimh{J_{\mathcal{F}}} \leq \dimh{J_{\mathcal{G}}}$.
        \item If $\mathcal{G}$ is irregular, then $\dimh{J_{\mathcal{F}}} \geq \dimh{J_{\mathcal{G}}}$.
    \end{enumerate}
\end{mthm}

\begin{mthm}[\Cref{thm:main_theorem_dim1_T1}]\label{main5}
    Let $\mathbf{d} \in \mathcal{D}$ and $T=1$.
    Suppose that $\mathcal{F} = \mathcal{F}(\mathbf{d}, 1)$ and $\mathcal{G} = \mathcal{G}(\mathbf{d}, 1)$ are CIFSs.
    Also, suppose that $\mathcal{G}$ is a regular CIFS.
    Then, we have
    \begin{equation*}
        \dimh{J_{\mathcal{F}}} > \dimh{J_{\mathcal{G}}}.
    \end{equation*}
\end{mthm}

\begin{mthm}[\Cref{thm:comparison_of_dimensions_T_4}, \Cref{cor:comparison_of_dimensions_T_4_at_half}]\label{main6}
    Let $\mathbf{d} = \set{d_{n}}_{n=1}^{\infty} \in \mathcal{D}$ and $T=4$.
    Suppose that $\mathcal{F} = \mathcal{F}(\mathbf{d}, 4)$ is a CIFS and $\mathcal{G} = \mathcal{G}(\mathbf{d}, 4)$ is a regular CIFS.
    If $h_{\mathcal{G}} := \dimh{J_{\mathcal{G}}} \in (0, 1)$,
    then there exists a constant $L = L(h_{\mathcal{G}}) \geq 2$ such that the following holds.
    \begin{itemize}
        \item If $d_{1} \geq L$, then $\dimh{J_{\mathcal{F}}} > \dimh{J_{\mathcal{G}}}$ holds.
    \end{itemize}
    In addition, $L = L(h_{\mathcal{G}})$ is non-decreasing with respect to $h_{\mathcal{G}}$,
    and if $0 < h_{\mathcal{G}} \leq 1/2$, then $L(h_{\mathcal{G}}) = 2$ holds.
\end{mthm}

We obtain the following examples by applying the above theorems.
\Cref{main7} provides an example where \textbf{(C1)} occurs, which is established by \Cref{main2}.

\begin{mthm}[\Cref{ex:C1_satisfied_example}]\label{main7}
    Take $\gamma > 1$ and $T \in \mathbb{N}$ arbitrarily.
    For each $L \in \mathbb{N}$,
    define the sequence $\mathbf{d}^{(L)} := \set{(n+L)^{\gamma}}_{n=1}^{\infty}$.
    Then, if $L$ is large enough,
    both $\mathcal{F}(\mathbf{d}^{(L)}, T)$ and $\mathcal{G}(\mathbf{d}^{(L)}, T)$ are CIFSs.
    Moreover, we have
    \begin{equation*}
        H_{h_{\mathcal{F}\paren{\mathbf{d}^{(L)}, T}}}\paren{J_{\mathcal{F}\paren{\mathbf{d}^{(L)}, T}}} = H_{h_{\mathcal{G}\paren{\mathbf{d}^{(L)}, T}}}\paren{J_{\mathcal{G}\paren{\mathbf{d}^{(L)}, T}}} = 0.
    \end{equation*}
\end{mthm}

The following \Cref{main8} provides an example satisfying \textbf{(C2)} with $\dimh{J_{\mathcal{F}}} > \dimh{J_{\mathcal{G}}}$.

\begin{mthm}[\Cref{ex:c2-and-comparison_of_dimensions_T4_example}]\label{main8}
    Define $\mathbf{d} = \set{d_{n}}_{n=1}^{\infty}$ by
    \begin{equation*}
        d_{n} = \begin{cases}
            17 & (n=1) \\
            19 & (n=2) \\
            n^{3} & (n \geq 3)
        \end{cases}.
    \end{equation*}
    Let $\mathcal{F} = \mathcal{F}(\mathbf{d}, 4), \ \mathcal{G} = \mathcal{G}(\mathbf{d}, 4)$.
    Then, the following hold.
    \begin{enumerate}
        \item $\dimp{J_{\mathcal{F}}} = \dimh{J_{\mathcal{F}}}$ and $\Pi_{\dimh{J_{\mathcal{F}}}}(J_{\mathcal{F}}) = \infty$.
        \item $\dimp{J_{\mathcal{G}}} = \dimh{J_{\mathcal{G}}}$ and $\Pi_{\dimh{J_{\mathcal{G}}}}(J_{\mathcal{G}}) = \infty$.
        \item $\dimh{J_{\mathcal{G}}} < \dimh{J_{\mathcal{F}}} < 1/2$.
    \end{enumerate}
\end{mthm}

\subsection{Organization of This Paper}
In \Cref{sec:CIFS}, we review the definitions and fundamental properties of infinite conformal iterated function systems (CIFSs)
based on \cite{MauldinUrbanski1996} and \cite{Takemoto-en}.
In particular, we introduce the pressure functions, regularity, the asymptotic boundary, and the relationship between dimensions of the limit sets and the measure-theoretic properties.
In \Cref{sec:f-and-g}, we define the IFSs $\mathcal{F}$ and $\mathcal{G}$,
which extend the IFSs of real continued fractions studied in \cite{MauldinUrbanski1999}
to the IFSs on the closed unit disc in the complex plane.
After investigating some fundamental properties of $\mathcal{F}$ and $\mathcal{G}$,
we prove \Cref{main1}.
In \Cref{sec:relationship_between_increasing_rate_and_measures},
we consider the case where the sequence $\mathbf{d} = \set{d_{n}}_{n=1}^{\infty}$ has polynomial growth with respect to $n$,
as in \cite{MauldinUrbanski1999},
and we prove \Cref{main2,main3}.
In \Cref{sec:comparison-of-dimension-f-and-g}, we prove \Cref{main4,main5,main6},
which are the results comparing the Hausdorff dimensions of the limit sets of $\mathcal{F}$ and $\mathcal{G}$.
In \Cref{sec:examples}, we give some examples illustrating the theorems proved in the preceding sections.

\section{Conformal Iterated Function Systems}\label{sec:CIFS}
In this section, we review the definitions and fundamental properties of infinite conformal iterated function systems (CIFSs).
Our exposition follows \cite{MauldinUrbanski1996} and \cite{Takemoto-en}.

Unless otherwise stated, we assume that $I$ is a countable (possibly finite) set with at least two elements
and that $X$ is a compact connected subset of $\mathbb{R}^{d}$.

\subsection{Definition of Infinite Conformal Iterated Function Systems and Limit Sets}
\begin{definition}[Iterated Function System]
  Let $\mathcal{S} = \set{s_{i}: X \to X}_{i \in I}$ be a family of injective contraction maps on $X$.
  We call $\mathcal{S}$ an \textit{iterated function system} (IFS) on $X$
  if $\sup_{i \in I} \Lip(s_{i}) < 1$.
\end{definition}

Throughout this section,
we assume that $\mathcal{S} = \set{s_{i}: X \to X}_{i \in I}$ is an IFS on $X$.

\begin{definition}
    \begin{enumerate}
        \item For $N\in\mathbb{N}, \omega = (\omega_{1}, \omega_{2}, \dots, \omega_{N}) \in I^{N}$,
                define $s_{\omega} : X \to X$ by
                $s_{\omega} := s_{\omega_{1}} \circ s_{\omega_{2}} \circ \dots \circ s_{\omega_{N}}$.
        \item For $\omega =  (\omega_{1}, \omega_{2}, \dots) \in I^{\mathbb{N}}, N\in\mathbb{N}$,
                define $\omega|_{N} \in I^{N}$ by $\omega|_{N} := (\omega_{1}, \omega_{2}, \dots, \omega_{N})$
    \end{enumerate}
\end{definition}

\begin{definition}[Limit set]
    The limit set of $\mathcal{S}$, denoted by $J_{\mathcal{S}}$, is defined by
    \begin{equation*}
        J_{\mathcal{S}} := \bigcup_{\omega \in I^{\mathbb{N}}} \bigcap_{N=1}^{\infty} s_{\omega|_{N}}(X).
    \end{equation*}
\end{definition}

\begin{prop}[Self-similarity of the limit set]\label{prop:self-similarity-of-limit-set}
    We have
    \begin{equation*}
        J_{\mathcal{S}} = \bigcup_{i \in I} s_{i}(J_{\mathcal{S}}).
    \end{equation*}
\end{prop}

\begin{definition}\label{def:S_x}
    For $x\in X$, we define
    \begin{equation*}
        \mathcal{S}_{x} := \set{s_{i}(x) \mid i \in I}.
    \end{equation*}
\end{definition}

\begin{definition}[Asymptotic boundary]
    The asymptotic boundary of $J_{\mathcal{S}}$, denoted by $X_{\mathcal{S}}(\infty)$, is defined as follows.
    \begin{equation*}
        X_{\mathcal{S}}(\infty) := \left\{ x \in X \;\middle|\;
        \begin{aligned}
            & \exists \text{ infinite subset } I' \subset I, \\
            & \exists x_{i} \in s_{i}(X) \ (i \in I') \text{ such that } x = \lim_{i\in I'} x_{i}
        \end{aligned}
        \right\}.
    \end{equation*}
    Here, $x = \lim_{i\in I'}x_{i}$ means that for any $\varepsilon > 0$,
    there exists a finite subset $F \subset I'$ such that
    for all $i \in I' \setminus F$, we have $|x - x_{i}| < \varepsilon$.
\end{definition}

\begin{rmk}
    \begin{enumerate}
        \item If $I$ is a finite set, then $X_{\mathcal{S}}(\infty) = \emptyset$.
        \item If $I$ is an infinite set, then $X_{\mathcal{S}}(\infty) \ne \emptyset$ by compactness of $X$.
        \item If $\lim_{i\in I}\diam\paren{s_{i}(X)} = 0$, then for each $x\in X$,
            we have
            \begin{equation*}
                X_{\mathcal{S}}(\infty) = \mathrm{Acc}(\mathcal{S}_{x}),
            \end{equation*}
            where $\mathrm{Acc}(\mathcal{S}_{x})$ denotes the set of accumulation points of $\mathcal{S}_{x}$.
    \end{enumerate}
\end{rmk}

\begin{prop}[\cite{MauldinUrbanski1996} Lemma 2.1]\label{prop:closure-of-limit-set}
    If $\lim_{i\in I}\diam\paren{s_{i}(X)} = 0$, then the following equality holds.
    \begin{equation*}
        \overline{J_{\mathcal{S}}} = J_{\mathcal{S}} \cup X_{\mathcal{S}}(\infty) \cup \bigcup_{N\in\mathbb{N}}\bigcup_{\omega \in I^{N}}s_{\omega}(X_{\mathcal{S}}(\infty)).
    \end{equation*}
\end{prop}

\begin{definition}[Conformal Iterated Function System (CIFS)]\label{def:CIFS}
    We say that an IFS
    \begin{equation*}
        \mathcal{S} = \set{s_{i} : X \to X}_{i \in I}
    \end{equation*}
    is a conformal iterated function system (CIFS)
    if the following four conditions are satisfied.
    \begin{description}[style=nextline, leftmargin=1.5em, font=\bfseries]
        \item[Open Set Condition (OSC)] For any $i, j \in I$ with $ i \ne j$, we have $s_{i}\paren{\Int(X)} \cap s_{j}\paren{\Int(X)} = \emptyset$.
        \item[Cone Condition (CC)] There exist constants $0 < \alpha < \pi/2, l > 0$ such that for every $x \in \partial X \subset \mathbb{R}^{d}$
                there exists $0 \ne u_{x} \in \mathbb{R}^{d}$ with $\Con(x, u_{x}, \alpha, l) \subset \Int(X)$.
        \item[Conformality] There exists an open connected subset $V$ of $\mathbb{R}^{d}$ containing $X$ such that
                            all maps $s_i (i \in I)$ extend to $C^{1 +\varepsilon}$ diffeomorphisms on $V$ and are conformal on $V$.
        \item[Bounded Distortion Property (BDP)] There exists $K \geq 1$ such that $\abs{s_{\omega}'(y)} \leq K \abs{s_{\omega}'(x)}$
            for every $\omega \in \bigcup_{N\in\mathbb{N}} I^{N}$ and $x, y \in V$,
                where $\abs{s_{\omega}'(x)}$ is the operator norm of the derivative of $s_{\omega}$ at $x$.
    \end{description}
\end{definition}

\begin{rmk}
    If \textbf{(CC)} holds, then $\Int(X)$ is obviously non-empty.
    Conversely, if $\Int(X) \ne \emptyset$ and $X$ is convex,
    then \textbf{(CC)} holds (\cite{MauldinUrbanski1996}).
\end{rmk}

\begin{definition}\label{def:approx-limit-set}
    For $N \in\mathbb{N}$, we define $J_{\mathcal{S}}^{(N)}$ by
    \begin{equation*}
        J_{\mathcal{S}}^{(N)} := \bigcup_{\omega \in I^{N}} s_{\omega}(X).
    \end{equation*}
\end{definition}

\begin{lem}[\cite{MauldinUrbanski1996} Lemma 2.6, (2.5)]\label{lem:Borel-measurability-of-limit-set}
    If $\mathcal{S}$ is a CIFS, then
    \begin{equation*}
        J_{\mathcal{S}} = \bigcap_{N=1}^{\infty} J_{\mathcal{S}}^{(N)}
    \end{equation*}
    holds. In particular, $J_{\mathcal{S}}$ is a Borel measurable set of $X$.
\end{lem}

We define a stronger separation condition than the open set condition.
\begin{definition}[Strong Open Set Condition (SOSC)]
    We say $\mathcal{S}$ satisfies the strong open set condition (SOSC) if
    \begin{equation*}
        J_{\mathcal{S}} \cap \Int(X) \ne \emptyset
    \end{equation*}
    holds.
\end{definition}

\subsection{Pressure Function and Regularity}

Throughout this section, unless otherwise stated,
we assume that $\mathcal{S} = \set{s_{i}: X \to X}_{i \in I}$ is a CIFS on $X$.

\begin{definition}[Pressure Function]\label{def:pressure_function}
    Define $P_{\mathcal{S}}: [0, \infty) \to \mathbb{R} \cup \set{\infty}$ by
    \begin{equation*}
        P_{\mathcal{S}}(t) := \lim_{N\to\infty} \frac{1}{N} \log\paren{\sum_{\omega \in I^{N}} \sup_{x\in X}\abs{s_{\omega}'(x)}^{t}} = \inf_{N\in\mathbb{N}} \frac{1}{N} \log\paren{\sum_{\omega \in I^{N}} \sup_{x\in X}\abs{s_{\omega}'(x)}^{t}} \quad (t \geq 0).
    \end{equation*}
    We call $P_{\mathcal{S}}$ the \textit{pressure function} of $\mathcal{S}$.
\end{definition}

The following bounds follow from the chain rule.

\begin{lem}\label{lem:pressure_bounds_rough_estimate}
    For $t \geq 0$, the following inequalities hold.
    \begin{equation*}
        \log\paren{\sum_{i \in I} \inf_{x\in X}\abs{s_{i}'(x)}^{t}} \leq P_{\mathcal{S}}(t) \leq \log\paren{\sum_{i \in I} \sup_{x\in X}\abs{s_{i}'(x)}^{t}}.
    \end{equation*}
\end{lem}

Moreover, (BDP) provides other expressions for the pressure function.

\begin{lem}\label{lem:other-expressions-of-pressure-function}
    For each $x \in X$ and $t \geq 0$, we have
    \begin{equation*}
        P_{\mathcal{S}}(t) = \lim_{N\to\infty} \frac{1}{N} \log\paren{\sum_{\omega \in I^{N}} \abs{s_{\omega}'(x)}^{t}}.
    \end{equation*}
    Moreover, we have
    \begin{equation*}
        P_{\mathcal{S}}(t) = \lim_{N\to\infty} \frac{1}{N} \log\paren{\sum_{\omega \in I^{N}} \inf_{x\in X}\abs{s_{\omega}'(x)}^{t}} \quad (t \geq 0).
    \end{equation*}
\end{lem}

\begin{definition}
    Define $F(\mathcal{S})$ and $\theta_{\mathcal{S}}$ by
    \begin{equation*}
        F(\mathcal{S}) := \set{t \geq 0 \mid P_{\mathcal{S}}(t) < \infty}, \quad \theta_{\mathcal{S}} := \inf F(\mathcal{S}).
    \end{equation*}
\end{definition}

\begin{prop}[\cite{MauldinUrbanski1996} Proposition 3.3]\label{prop:properties-of-pressure-function}
    $P_{\mathcal{S}}$ satisfies the following properties.
    \begin{enumerate}
        \item $P_{\mathcal{S}}$ is non-increasing on $[0, \infty)$.
        \item $P_{\mathcal{S}}$ is strictly decreasing on $[\theta_{\mathcal{S}}, \infty)$.
        \item $P_{\mathcal{S}}$ is continuous and convex on $F(\mathcal{S})$.
    \end{enumerate}
\end{prop}

\begin{definition}\label{def:regularity}
    \begin{enumerate}
        \item We say that $\mathcal{S}$ is regular if there exists $t \geq 0$ such that $P_{\mathcal{S}}(t) = 0$.
        \item We say that $\mathcal{S}$ is irregular if it is not regular.
        \item We say that $\mathcal{S}$ is hereditarily regular if for any subset $I' \subset I$ with $\# I' \geq 2$ and $\#\paren{I\setminus I'} < \infty$,
                the subsystem $\set{s_{i}}_{i \in I'}$ is regular.
    \end{enumerate}
\end{definition}

\begin{prop}\label{prop:finite-IFS-is-hereditarily-regular}
    If $\# I < \infty$, then $\mathcal{S}$ is hereditarily regular.
\end{prop}

\begin{proof}
    Suppose that $2 \leq \# I < \infty$.
    We have $0 < P_{\mathcal{S}}(0) = \log(\# I) < \infty$ by the definition of $P_{\mathcal{S}}$.
    We also have $\lim_{t \to \infty} P_{\mathcal{S}}(t) = -\infty$
    by $\sup_{i \in I} \Lip(s_{i}) < 1$, conformality of each $s_{i}$ and the cone condition.
    Therefore, since $P_{\mathcal{S}}$ is continuous on $[0, \infty)$ (\Cref{prop:properties-of-pressure-function} (3)),
    $P_{\mathcal{S}}$ has a positive zero.
    Thus, $\mathcal{S}$ is regular.
    Similarly, for any subset $I' \subset I$ with $\# I' \geq 2$ and $\#\paren{I\setminus I'} < \infty$,
    the subsystem $\set{s_{i}}_{i \in I'}$ is also regular.
    Hence $\mathcal{S}$ is hereditarily regular.
\end{proof}

\begin{thm}[\cite{MauldinUrbanski1996} Theorem 3.20]\label{thm:hereditarily-regularity-and-pressure}
    If $\#I = \infty$, then the following are equivalent.
    \begin{enumerate}
        \item $\mathcal{S}$ is hereditarily regular.
        \item $P_{\mathcal{S}}(\theta_{\mathcal{S}}) = \infty$.
    \end{enumerate}
\end{thm}

The following theorems describe the relationship between the Hausdorff dimension of the limit set and the pressure function.

\begin{thm}[\cite{MauldinUrbanski1996} Theorem 3.15]\label{thm:hausdorff-dimension-and-pressure}
    We have
    \begin{equation*}
        \dimh{J_{\mathcal{S}}} = \inf\set{t \geq 0 \mid P_{\mathcal{S}}(t) < 0}.
    \end{equation*}
    In particular, if $t > 0$ satisfies $P_{\mathcal{S}}(t)=0$, then
    $t=\dimh{J_{\mathcal{S}}}$.
    Moreover, we have $\theta_{\mathcal{S}} \leq \dimh{J_{\mathcal{S}}}$.
\end{thm}

\begin{thm}[\cite{MauldinUrbanski1996} Theorem 3.20]\label{thm:theta-and-dimh-if-hereditarily-regular}
    If $\mathcal{S}$ is hereditarily regular, then $\theta_{\mathcal{S}} < \dimh{J_{\mathcal{S}}}$.
\end{thm}

\subsection{Measure and Dimension of the Limit Sets}
In this subsection, we review the measure-theoretic properties of the limit sets $J_{\mathcal{S}}$ of CIFSs $\mathcal{S}$.
If $\mathcal{S}$ is a regular CIFS, then there exists a probability measure on $J_{\mathcal{S}}$ called conformal measure.
This measure plays important roles in investigating the Hausdorff and packing dimensions of $J_{\mathcal{S}}$.

\begin{definition}\label{def:conformal-measure}
    Let $m$ be a Borel probability measure on $X$ and let $t > 0$.
    We say that $m$ is a $t$-conformal measure for $\mathcal{S}$
    if the following (1), (2), and (3) hold.
    \begin{enumerate}
        \item $m\paren{J_{\mathcal{S}}} = 1$.
        \item For each $\omega \in \bigcup_{N\in\mathbb{N}}I^{N}$ and for each Borel set $A$ on $J_{\mathcal{S}}$, we have
            \begin{equation*}
                m\paren{s_{\omega}(A)} = \int_{A} \abs{s_{\omega}'(x)}^{t} dm(x).
            \end{equation*}
        \item For each $i, j \in I$ with $i \ne j$, we have $m\paren{s_{i}(X) \cap s_{j}(X)} = 0$.
    \end{enumerate}
\end{definition}

\begin{thm}[\cite{MauldinUrbanski1996} Theorem 3.18]\label{thm:existence-and-uniqueness-of-conformal-measure}
    \begin{enumerate}
        \item For each $t>0$, there exists at most one $t$-conformal measure for $\mathcal{S}$.
        \item For each $t>0$, $P_{\mathcal{S}}(t)=0$ if and only if there exists a $t$-conformal measure for $\mathcal{S}$.
    \end{enumerate}
\end{thm}

Throughout this paper, we often denote $\dimh{J_{\mathcal{S}}}$ by $h_{\mathcal{S}}$ for simplicity.

\begin{thm}[\cite{MauldinUrbanski1996} Theorem 3.18]\label{thm:existence-and-uniqueness-of-conformal-measure-for-hausdorff-dimension}
    Suppose $\mathcal{S}$ is a regular CIFS.
    Then,
    $P_{\mathcal{S}}(h_{\mathcal{S}}) = 0$ holds and there exists a unique
    $h_{\mathcal{S}}$-conformal measure for $\mathcal{S}$.
\end{thm}

Therefore, for a regular CIFS $\mathcal{S}$,
we denote the $h_{\mathcal{S}}$-conformal measure for $\mathcal{S}$
by $m_{\mathcal{S}}$.

If $\mathcal{S}$ is an infinite regular CIFS,
then the density at a point in $X_{\mathcal{S}}(\infty)$ with respect to the $h_{\mathcal{S}}$-conformal measure
has the following relationship with the Hausdorff and packing measures of $J_{\mathcal{S}}$.

\begin{thm}[\cite{MauldinUrbanski1996} Lemma 4.9]\label{thm:hausdorff-measure-x-infty}
    Let $\mathcal{S}$ be a regular CIFS.
    If there exists a sequence $\{z_{j}\}_{j\in\mathbb{N}}$ in $X_{\mathcal{S}}(\infty)$ and
    a sequence of positive real numbers $\{r_{j}\}_{j\in\mathbb{N}}$ such that
    \begin{equation*}
        \limsup_{j\to\infty}\frac{m_{\mathcal{S}}(B(z_{j}, r_{j}))}{r_{j}^{h_{\mathcal{S}}}} = \infty,
    \end{equation*}
    then $H_{h_{\mathcal{S}}}(J_{\mathcal{S}}) = 0$.
\end{thm}

\begin{cor}\label{cor:hausdorff-measure-x-infty}
    If there exists $z \in X_{\mathcal{S}}(\infty)$ such that
    \begin{equation*}
        \limsup_{r\downarrow 0}\frac{m_{\mathcal{S}}(B(z, r))}{r^{h_{\mathcal{S}}}} = \infty,
    \end{equation*}
    then $H_{h_{\mathcal{S}}}(J_{\mathcal{S}}) = 0$.
\end{cor}

The proof of Lemma 4.13 in \cite{MauldinUrbanski1996} has an error.
However, the following corrected version of the lemma holds by changing the assumption slightly.

\begin{thm}\label{thm:packing-measure-x-infty}
    Let $\mathcal{S}$ be a regular CIFS.
    If there exists $z\in X_{\mathcal{S}}(\infty)\cap \Int(X)$ such that
    \begin{equation*}
        \limsup_{r\downarrow 0}\frac{m_{\mathcal{S}}(B(z, r))}{r^{h_{\mathcal{S}}}} = 0,
    \end{equation*}
    then $\Pi_{h_{\mathcal{S}}}(J_{\mathcal{S}}) = \infty$.
\end{thm}

\begin{proof}
    Since $z \in \Int(X)$, there exists $r > 0$ such that $B(z, r) \subset X$.
    Since $z \in X_{\mathcal{S}}(\infty)$,
    there exists an infinite subset $I' \subset I$ and $x_{i} \in s_{i}(X) \ (i \in I')$
    such that $z = \lim_{i\in I'} x_{i}$.
    Note that $\lim_{i\in I'}\diam(s_{i}(X)) = 0$ by (OSC) and (BDP) (see \cite{MauldinUrbanski1996} Lemma 2.5 for details).
    Passing to a further infinite subset of $I'$ if necessary, we may assume that
    \begin{equation*}
        \sup_{i\in I'}|x_{i} - z| < r/4, \quad \sup_{i\in I'}\diam(s_{i}(X)) < r/4.
    \end{equation*}
    Fix $z_{i} \in s_{i}(X)$ for each $i \in I'$.
    Let $I'' := \set{i \in I' \mid |z - z_{i}| > 0}$.
    Then, $I' \setminus I''$ is finite (see \cite{MauldinUrbanski1996} Lemma 2.6 for details),
    and hence $I''$ is infinite.

    For each $i \in I''$, if $x \in B(z_{i}, \abs{z - z_{i}})$, then we have
    \begin{align*}
        |z - x| &\leq |z - z_i| + |z_{i} - x| \\
        &< 2|z - z_i| \\
        &\leq 2(|z - x_i| + |x_i - z_{i}|) \\
        &< 2(r/4 + r/4) = r.
    \end{align*}
    These inequalities imply that
    \begin{equation*}
        B(z_i, |z - z_i|) \subset B(z, 2|z - z_i|) \subset B(z, r) \subset X.
    \end{equation*}
    Therefore, for each $i \in I''$, we have
    \begin{equation*}
        \frac{m_{\mathcal{S}}(B(z_{i}, |z - z_{i}|))}{|z - z_{i}|^{h_{\mathcal{S}}}} \leq 2^{h_{\mathcal{S}}}\cdot\frac{m_{\mathcal{S}}(B(z, 2|z - z_{i}|))}{(2|z - z_{i}|)^{h_{\mathcal{S}}}}.
    \end{equation*}

    Let $\eta > 0$. By the assumptions of the theorem,
    there exists $\varepsilon > 0$ such that
    \begin{equation*}
        \sup_{0 < r < \varepsilon}\frac{m_{\mathcal{S}}(B(z, r))}{r^{h_{\mathcal{S}}}} < \frac{\eta}{2^{h_{\mathcal{S}}}}.
    \end{equation*}
    Hence, for each $i \in I''$ with $2\abs{z - z_{i}} < \varepsilon$, we have
    \begin{equation*}
        \frac{m_{\mathcal{S}}(B(z_{i}, |z - z_{i}|))}{|z - z_{i}|^{h_{\mathcal{S}}}} < \eta.
    \end{equation*}
    Since $\lim_{i\in I''}|z - z_{i}| = 0$, it follows that
    \begin{equation*}
        \lim_{i\in I''}\frac{m_{\mathcal{S}}(B(z_{i}, |z - z_{i}|))}{|z - z_{i}|^{h_{\mathcal{S}}}} = 0.
    \end{equation*}
    Therefore, by Lemma 4.12 in \cite{MauldinUrbanski1996}, we conclude that $\Pi_{h_{\mathcal{S}}}\paren{J_{\mathcal{S}}} = \infty$.
\end{proof}

\begin{rmk}
    Let $\mathcal{S} = \set{s_{i} : X \to X}_{i \in I}$ be a CIFS.
    If $I$ is finite, the following holds (see \cite{MauldinUrbanski1996} Lemma 3.14).
    \begin{equation*}
        h_{\mathcal{S}} = \dimh{J_{\mathcal{S}}} = \dimp{J_{\mathcal{S}}}, \quad 0 < H_{h_{\mathcal{S}}}(J_{\mathcal{S}}) < \infty, \quad 0 < \Pi_{h_{\mathcal{S}}}(J_{\mathcal{S}}) < \infty.
    \end{equation*}
    Consequently, for regular CIFSs, the properties \textbf{(C1)} and \textbf{(C2)} in \Cref{sec:introduction},
    which appear in the conclusions of \Cref{thm:hausdorff-measure-x-infty} and \Cref{thm:packing-measure-x-infty},
    can occur only when $I$ is infinite.
\end{rmk}

\begin{rmk}
    It is known that the following hold for a regular CIFS $\mathcal{S}$ (see Lemmas 4.2 and 4.3 of \cite{MauldinUrbanski1996}).
    \begin{enumerate}
        \item $H_{h_{\mathcal{S}}}(J_{\mathcal{S}}) < \infty$ .
        \item If $\# I < \infty$ or (SOSC) holds,
            then $\Pi_{h_{\mathcal{S}}}(J_{\mathcal{S}}) > 0$.
    \end{enumerate}
\end{rmk}

At the end of this section, we review the relationship between the Hausdorff, packing, and box dimensions of the limit sets of CIFSs.

\begin{lem}[\cite{Falconer} p.56]
    Let $A$ be a non-empty bounded set of $\mathbb{R}^{d}$. Then,
    \begin{equation*}
        \dimh{A} \leq \dimp{A} \leq \dimub{A}.
    \end{equation*}
\end{lem}

\begin{thm}[\cite{MauldinUrbanski1999} Theorem 2.11]\label{thm:packing-dimension-and-box-dimension}
    Let $\mathcal{S} = \set{s_{i}:X \to X}_{i \in I}$ be a CIFS.
    Then, for each $x\in X$,
    \begin{equation*}
        \dimp{J_{\mathcal{S}}} = \max\set{\dimh{J_{\mathcal{S}}}, \dimub{\mathcal{S}_{x}}}.
    \end{equation*}
\end{thm}

\begin{thm}[\cite{MauldinUrbanski1996} Theorem 3.1]\label{thm:box-dimension-and-packing-dimension-biLipschitz-case}
    Let $\mathcal{S} = \set{s_{i}:X \to X}_{i \in I}$ be a CIFS.
    Suppose that each $s_{i}$ is bi-Lipschitz continuous on $X$.
    Then,
    \begin{equation*}
        \dimp{J_{\mathcal{S}}} = \dimub{J_{\mathcal{S}}} = \dimp{\overline{J_{\mathcal{S}}}} = \dimub{\overline{J_{\mathcal{S}}}}.
    \end{equation*}
\end{thm}

\section{Definition and fundamental properties of \texorpdfstring{$\mathcal{F}$ and $\mathcal{G}$}{F and G}}\label{sec:f-and-g}
In this section, we define the main objects of this paper, the IFSs $\mathcal{F}$ and $\mathcal{G}$,
and we study their fundamental properties.

\subsection{Definition and Fundamental Properties of \texorpdfstring{$f_{\alpha}$ and $g_{\alpha}$}{f and g}}
We first introduce the maps $f_{\alpha}$ and $g_{\alpha}$ that generate the IFSs $\mathcal{F}$ and $\mathcal{G}$,
and we investigate their basic properties.

\begin{definition}\label{def:f-and-g}
    Let $\alpha$ be a complex number with $\abs{\alpha} \geq 2$.
    We define the maps $f_{\alpha}, g_{\alpha} : \overline{\mathbb{D}} \to \overline{\mathbb{D}}$ by
    \begin{equation*}
        f_{\alpha}(z) := \frac{z + \overline{\alpha}}{\abs{\alpha}^{2}-1}, \quad g_{\alpha}(z) := \frac{1}{z + \alpha}.
    \end{equation*}
\end{definition}

\begin{prop}\label{prop:image}
    \begin{enumerate}
        \item Let $\alpha \in \mathbb{C}$ with $\abs{\alpha} \geq 2$.
            Then,
            \begin{equation*}
                f_{\alpha}\paren{\overline{\mathbb{D}}} = g_{\alpha}\paren{\overline{\mathbb{D}}} = \overline{B}\paren{\frac{\overline{\alpha}}{|\alpha|^{2}-1}, \frac{1}{|\alpha|^{2}-1}}.
            \end{equation*}
        \item Let $\alpha_{1}, \alpha_{2}\in\mathbb{R}$ with $\alpha_{1} - \alpha_{2} \geq 2$ and $\alpha_{2} \geq 2$.
            Then,
            \begin{equation*}
                f_{\alpha_{1}}\paren{\overline{\mathbb{D}}} \cap f_{\alpha_{2}}\paren{\overline{\mathbb{D}}} = g_{\alpha_{1}}\paren{\overline{\mathbb{D}}} \cap g_{\alpha_{2}}\paren{\overline{\mathbb{D}}} = \begin{cases}
                    \emptyset. & (\alpha_{1} - \alpha_{2} > 2) \\
                    \left\{\frac{1}{\alpha_{1}-1}\right\} = \left\{\frac{1}{\alpha_{2}+1}\right\} & (\alpha_{1} - \alpha_{2} = 2).
                \end{cases}
            \end{equation*}
    \end{enumerate}
\end{prop}

\begin{definition}
    Let $N\in\mathbb{N}, \alpha_{1}, \dots, \alpha_{N} \in \mathbb{C}$.
    Assume that for each $j=1, 2, \dots, N$, $\abs{\alpha_{j}} \geq 2$.
    We define the continued fraction $[\alpha_{1}, \alpha_{2}, \dots, \alpha_{N}]$ by
    \begin{equation*}
        [\alpha_{1}, \alpha_{2}, \dots, \alpha_{N}] := \cfrac{1}{\alpha_{1} + \cfrac{1}{\alpha_{2} + \cfrac{1}{\ddots + \cfrac{1}{\alpha_{N}}}}}.
    \end{equation*}
\end{definition}

\begin{prop}\label{prop:derivative}
    Assume that $N\in\mathbb{N}, \ \alpha_{1}, \dots, \alpha_{N} \in \mathbb{C}$ and
    $\abs{\alpha_{j}} \geq 2$ for each $j=1, 2, \dots, N$.
    Then, for each $z \in \overline{\mathbb{D}}$,
    \begin{align*}
        (f_{\alpha_{1}} \circ \dots \circ f_{\alpha_{N}})'(z) &= \prod_{j=1}^{N}\frac{1}{(|\alpha_{j}|^{2}-1)}, \\
        (g_{\alpha_{1}} \circ \dots \circ g_{\alpha_{N}})'(z) &= (-1)^{N}\prod_{j=1}^{N}[\alpha_{j}, \alpha_{j+1}, \dots, \alpha_{N}+z]^{2}.
    \end{align*}
\end{prop}

\begin{lem}
    Let $\alpha \in \mathbb{C}$.
    \begin{enumerate}
        \item If $\abs{\alpha} \geq 2$, then we have that $\Lip(f_{\alpha}) = \frac{1}{\abs{\alpha}^{2}-1} \leq 1/3$.
        \item If $\abs{\alpha} > 2$, then we have that $\Lip(g_{\alpha}) = \frac{1}{(\abs{\alpha}-1)^{2}}<1$.
    \end{enumerate}
\end{lem}

\begin{proof}
    The part (1) follows from the definition of $f_{\alpha}$.
    For any $z, w \in \overline{\mathbb{D}}$, we have
    \begin{align*}
        \left|g_{\alpha}(z) - g_{\alpha}(w)\right| &= \left|\frac{1}{z+\alpha} - \frac{1}{w+\alpha}\right| \\
        &= \frac{|z-w|}{|z+\alpha||w+\alpha|} \leq \frac{|z-w|}{(|\alpha| - 1)^{2}}.
    \end{align*}
    Hence the part (2) holds.
\end{proof}

\begin{lem}\label{lem:extension-to-bigger-domain}
    Let $\alpha \in \mathbb{C}$.
    \begin{enumerate}
        \item Let $|\alpha| \geq 2$ and let $V = B(0, 2)$. Then, $f_{\alpha}$ extends to a conformal $C^{\infty}$-diffeomorphism on $V$ to its image.
        \item Let $|\alpha| \geq r > 2$ and let $0 \leq \varepsilon < \sqrt{r^{2}-4}$.
        Define $V_{\varepsilon} = B\paren{0, \frac{r+\varepsilon}{2}}$.
        Then, $g_{\alpha}$ extends to a conformal $C^{\infty}$-diffeomorphism on $V_{\varepsilon}$ to its image.
    \end{enumerate}
\end{lem}

\begin{proof}
    (1) is trivial.
    As for (2), it is straightforward to see that $g_{\alpha}$ is a conformal $C^{\infty}$-diffeomorphism on the larger domain $V_{\varepsilon}$ onto its image.
    Moreover, we have
    \begin{equation*}
        g_{\alpha}(V_{\varepsilon}) =\overline{B}\paren{\frac{\overline{\alpha}}{|\alpha|^{2} - ((r+\varepsilon)/2)^{2}}, \frac{(r+\varepsilon)/2}{|\alpha|^{2} - ((r+\varepsilon)/2)^{2}}}.
    \end{equation*}
    Thus, for all $z\in V_{\varepsilon}$, we obtain
    \begin{align*}
        |g_{\alpha}(z)| &\leq \frac{|\alpha|}{|\alpha|^{2} - ((r+\varepsilon)/2)^{2}} + \frac{(r+\varepsilon)/2}{|\alpha|^{2} - ((r+\varepsilon)/2)^{2}} \\
        &= \frac{1}{|\alpha| - (r+\varepsilon)/2} \\
        &\leq \frac{1}{r-(r+\varepsilon)/2} = \frac{(r+\varepsilon)/2}{(r^{2} - \varepsilon^{2}) / 4} < \frac{r+\varepsilon}{2}.
    \end{align*}
    Hence, $g_{\alpha}(V_{\varepsilon}) \subset V_{\varepsilon}$.
\end{proof}

We will use the following theorem to prove that $\mathcal{F}$ and $\mathcal{G}$ satisfy the BDP.

\begin{thm}[Koebe's distortion theorem (\cite{Pommerenke1992} Theorem 1.3, \cite{Sugita-en} Theorem 6.9)]\label{thm:koebe_distortion_theorem}
    Let $\varphi: \mathbb{D} \to \mathbb{C}$ be a univalent holomorphic function with $\varphi(0) = 0$ and $\varphi'(0) = 1$.
    Then, for all $z \in \mathbb{D}$,
    \begin{equation*}
        \frac{1-|z|}{(1+|z|)^{3}} \leq |\varphi'(z)| \leq \frac{1+|z|}{(1-|z|)^{3}}.
    \end{equation*}
\end{thm}

\begin{lem}\label{lem:koebe_for_g}
    Let $r > 2$. For $0 <\varepsilon <\sqrt{r^{2} - 4}$,
    set $V_{\varepsilon} := B\paren{0, \frac{r+\varepsilon}{2}}$.
    Also, let $N \in \mathbb{N}$
    and suppose that $\alpha_{1}, \dots, \alpha_{N} \in \mathbb{C}$
    satisfy $\abs{\alpha_{j}} \geq r$ for each $j=1, \dots, N$.
    Define $g := g_{\alpha_{1}} \circ \dots \circ g_{\alpha_{N}}$.
    Then, for each $z\in V_{\varepsilon}$, we have
    \begin{equation*}
        \frac{1-\left|\frac{2z}{r+\varepsilon}\right|}{\paren{1+\left|\frac{2z}{r+\varepsilon}\right|}^{3}} \leq \frac{|g'(z)|}{|g'(0)|} \leq \frac{1+\left|\frac{2z}{r+\varepsilon}\right|}{\paren{1-\left|\frac{2z}{r+\varepsilon}\right|}^{3}}.
    \end{equation*}
\end{lem}

\begin{proof}
    Define $\varphi : \mathbb{D} \to \mathbb{C}$ by
    \begin{equation*}
        \varphi(z) := \paren{g\paren{\frac{r+\varepsilon}{2}z} - g(0)}\frac{1}{g'(0)}\frac{2}{r+\varepsilon}.
    \end{equation*}
    Then, $\varphi$ is a univalent holomorphic function on $\mathbb{D}$ and satisfies
    $\varphi(0) = 0, \ \varphi'(0) = 1$.
    Therefore, by Koebe's distortion theorem, we obtain
    \begin{equation*}
        \frac{1-|z|}{(1+|z|)^{3}} \leq |\varphi'(z)| \leq \frac{1+|z|}{(1-|z|)^{3}}.
    \end{equation*}
    Note that $\varphi'(z) = g'\paren{\frac{r+\varepsilon}{2}z}/g'(0) \ (z\in\mathbb{D})$.
    This equation implies that $g'(z)/g'(0) = \varphi'\paren{\frac{2z}{r+\varepsilon}}$ holds
    for each $z\in V_{\varepsilon}$. Therefore, the lemma follows.
\end{proof}

\begin{lem}\label{lem:bdp}
    Let $N\in\mathbb{N}$ and let $\alpha_{1}, \dots, \alpha_{N} \in \mathbb{C}$.
    \begin{enumerate}
        \item Let $V = B(0, 2)$.
            Suppose that $\abs{\alpha_{j}} \geq 2$ for each $j=1, 2, \dots, N$.
            Define $f := f_{\alpha_{1}} \circ \dots \circ f_{\alpha_{N}}$.
            Then,
            \begin{equation*}
                \sup_{x, y \in V}\frac{|f'(x)|}{|f'(y)|} = 1.
            \end{equation*}
        \item Let $r>2$ and let $V = B\paren{0, \frac{r}{2}}$.
            Suppose that $\abs{\alpha_{j}} \geq r$ for each $j=1, 2, \dots, N$.
            Define $g := g_{\alpha_{1}} \circ \dots \circ g_{\alpha_{N}}$.
            Then, for any
            $0 < \varepsilon < \sqrt{r^{2} - 4}$,
            \begin{equation*}
                \sup_{x, y \in V}\frac{|g'(x)|}{|g'(y)|} \leq \paren{\frac{2r+\varepsilon}{\varepsilon}}^{4}.
            \end{equation*}
    \end{enumerate}
\end{lem}

\begin{proof}
    (1) is clear. It follows from \Cref{lem:koebe_for_g} that
    \begin{equation*}
        \sup_{x, y \in V}\frac{|g'(x)|}{|g'(y)|} \leq \sup_{x, y \in V}\frac{1+\left|\frac{2x}{r+\varepsilon}\right|}{\paren{1-\left|\frac{2x}{r+\varepsilon}\right|}^{3}}\frac{\paren{1+\left|\frac{2y}{r+\varepsilon}\right|}^{3}}{1-\left|\frac{2y}{r+\varepsilon}\right|} \leq \paren{\frac{1+\frac{r}{r+\varepsilon}}{1-\frac{r}{r+\varepsilon}}}^{4} = \paren{\frac{2r+\varepsilon}{\varepsilon}}^{4},
    \end{equation*}
    and hence (2) holds.
\end{proof}

\subsection{Definition of \texorpdfstring{$\mathcal{F}$ and $\mathcal{G}$}{F and G}}
In this subsection, we define the IFSs $\mathcal{F}$ and $\mathcal{G}$,
and show the fundamental properties that they satisfy.

\begin{definition}\label{def:d}
    We define the set $\mathcal{D}$ of sequences of real numbers as follows:
    \begin{equation*}
        \mathcal{D} := \set{\set{d_n}_{n=1}^{\infty} \mid d_1 \geq 2 \text{ and for all } n \in \mathbb{N}, d_{n+1} - d_n \geq 2}.
    \end{equation*}
\end{definition}

\begin{definition}\label{def:f-and-g-ifs}
    Let $\mathbf{d} = \set{d_n}_{n=1}^{\infty} \in \mathcal{D}$ and $T \in \mathbb{N}$.
    We define the families $\mathcal{F} = \mathcal{F}(\mathbf{d}, T)$ and $\mathcal{G} = \mathcal{G}(\mathbf{d}, T)$
    of maps on $\overline{\mathbb{D}}$ as follows.
    \begin{align*}
        \mathcal{F} = \mathcal{F}(\mathbf{d}, T) &:= \set{f_{e\paren{j/T}d_{n}} \mid n \in\mathbb{N}, j=0,1, \dots, T-1}, \\
        \mathcal{G} = \mathcal{G}(\mathbf{d}, T) &:= \set{g_{e\paren{j/T}d_{n}} \mid n \in\mathbb{N}, j=0,1, \dots, T-1}.
    \end{align*}
    Also, we call each $e(j/T)d_{n}$ a digit of $\mathcal{F}, \mathcal{G}$.
\end{definition}

\begin{lem}\label{lem:cifs-sufficient-cond}
    Let $\mathbf{d} = \set{d_n}_{n=1}^{\infty} \in \mathcal{D}$ and $T \in \mathbb{N}$.
    Suppose that
    \begin{equation*}
        T \leq \pi/\arctan(1/\sqrt{d_{1}^{2}-1}) = \pi/\arcsin(1/d_{1}).
    \end{equation*}
    Then, the following statements hold.
    \begin{enumerate}
        \item $\mathcal{F}\paren{\mathbf{d}, T}$ is a CIFS.
        \item If $d_{1} > 2$, then $\mathcal{G}\paren{\mathbf{d}, T}$ is a CIFS.
    \end{enumerate}
\end{lem}

\begin{proof}
    We show that the four conditions in \Cref{def:CIFS} hold for $\mathcal{F}\paren{\mathbf{d}, T}$ and $\mathcal{G}\paren{\mathbf{d}, T}$.
    It is clear that the cone condition holds.
    Conformality follows from \Cref{lem:extension-to-bigger-domain} and the BDP follows from \Cref{lem:bdp}.

    We now show that the open set condition (OSC) holds.
    By \Cref{prop:image}, for each $n \in \mathbb{N}$, we have
    \begin{equation*}
        f_{d_{n}}\paren{\overline{\mathbb{D}}} = g_{d_{n}}\paren{\overline{\mathbb{D}}} = \overline{B}\paren{\frac{d_{n}}{d_{n}^{2}-1}, \frac{1}{d_{n}^{2}-1}}.
    \end{equation*}
    Since $\overline{B}\paren{\frac{d_{n}}{d_{n}^{2}-1}, \frac{1}{d_{n}^{2}-1}}$ is tangent to
    two lines $\mathrm{Im}(z) = \pm \sqrt{\frac{1}{d_{n}^{2}-1}}\mathrm{Re}(z)$,
    it follows that for each $n \in \mathbb{N}$,
    $f_{d_{n}}\paren{\overline{\mathbb{D}}} = g_{d_{n}}\paren{\overline{\mathbb{D}}}$
    is contained in the region
    bounded by two lines $\mathrm{Im}(z) = \pm \sqrt{\frac{1}{d_{1}^{2}-1}}\mathrm{Re}(z)$,
    where the real part is positive.
    Noting this and $f_{pd_{n}}\paren{\overline{\mathbb{D}}} = \overline{p} f_{d_{n}}\paren{\overline{\mathbb{D}}} \ (|p| = 1)$,
    it follows that when $T \leq \pi/\arctan(1/\sqrt{d_{1}^{2}-1}) = \pi/\arcsin(1/d_{1})$, the OSC holds.
\end{proof}

\begin{rmk}
    Since we assume $d_{1} \geq 2$,
    we have $\arcsin(1/d_{1}) \leq \arcsin(1/2) = \pi/6$.
    Therefore, if $T\leq 6$, the assumption of \Cref{lem:cifs-sufficient-cond} always holds.
\end{rmk}

\begin{prop}\label{prop:comparison_of_f_and_g}
    Let $\mathbf{d} \in \mathcal{D}$ and $T \in \mathbb{N}$.
    Then, the following statements hold.
    \begin{enumerate}
        \item $J_{\mathcal{F}\paren{\mathbf{d}, T}}^{(1)} = J_{\mathcal{G}\paren{\mathbf{d}, T}}^{(1)}$.
        \item $X_{\mathcal{F}\paren{\mathbf{d}, T}}(\infty) = X_{\mathcal{G}\paren{\mathbf{d}, T}}(\infty) = \{0\}$.
    \end{enumerate}
\end{prop}

\begin{proof}
    Both (1) and (2) follow immediately from \Cref{prop:image} (1).
\end{proof}

\begin{rmk}
    As in Proposition 4.4 and Theorem 4.5 of \cite{MauldinUrbanski1996},
    for a general CIFS $\mathcal{S} = \set{s_{i}: X \to X}_{i \in I}$,
    there exist some properties of $J_{\mathcal{S}}$ that can be
    determined only by $J_{\mathcal{S}}^{(1)}$.
    These properties and \Cref{prop:comparison_of_f_and_g} (1)
    lead us to compare the properties of $\mathcal{F}\paren{\mathbf{d}, T}$
    and $\mathcal{G}\paren{\mathbf{d}, T}$.
\end{rmk}

\subsection{Pressure Functions of \texorpdfstring{$\mathcal{F}$ and $\mathcal{G}$}{F and G}}
In this subsection, we study the pressure functions of $\mathcal{F}$ and $\mathcal{G}$
and investigate their convergence.

\begin{lem}\label{lem:pressure_of_f}
    Let $\mathbf{d} = \set{d_n}_{n=1}^{\infty} \in \mathcal{D}$ and $T \in \mathbb{N}$.
    Suppose that $\mathcal{F}\paren{\mathbf{d}, T}$ is a CIFS. Then, for any $t \geq 0$,
    \begin{equation*}
        P_{\mathcal{F}\paren{\mathbf{d}, T}}(t) = \log\paren{T\sum_{n=1}^{\infty} \frac{1}{(d_{n}^{2}-1)^{t}}}.
    \end{equation*}
\end{lem}

\begin{proof}
    The statement follows from \Cref{prop:derivative}.
\end{proof}

\begin{lem}\label{lem:pressure_ineq_g}
    Let $\mathbf{d} = \set{d_n}_{n=1}^{\infty} \in \mathcal{D}$ and $T \in \mathbb{N}$.
    Suppose that $\mathcal{G}\paren{\mathbf{d}, T}$ is a CIFS. Then, for any $t \geq 0$,
    \begin{equation*}
        \log\paren{T\sum_{n=1}^{\infty} \frac{1}{(d_{n}+1)^{2t}}} \leq P_{\mathcal{G}\paren{\mathbf{d}, T}}(t) \leq \log\paren{T\sum_{n=1}^{\infty} \frac{1}{(d_{n}-1)^{2t}}}.
    \end{equation*}
\end{lem}

\begin{proof}
    Since $g_{\alpha}'(z) = -1/(z+\alpha)^{2}$ for $\alpha \in \mathbb{C}$ with $\abs{\alpha} \geq 2$, we have
    \begin{equation*}
        \sup_{z\in \overline{\mathbb{D}}}|g_{\alpha}'(z)| = \frac{1}{(|\alpha|-1)^{2}}, \quad \inf_{z\in \overline{\mathbb{D}}}|g_{\alpha}'(z)| = \frac{1}{(|\alpha|+1)^{2}}.
    \end{equation*}
    Therefore, by \Cref{lem:pressure_bounds_rough_estimate},
    we obtain
    \begin{align*}
        P_{\mathcal{G}\paren{\mathbf{d}, T}}(t) &\leq \log\paren{\sum_{j=0}^{T-1}\sum_{n=1}^{\infty} \frac{1}{\paren{\abs{e(j/T)d_{n}}-1}^{2t}}} = \log\paren{T\sum_{n=1}^{\infty} \frac{1}{(d_{n}-1)^{2t}}}, \\
        P_{\mathcal{G}\paren{\mathbf{d}, T}}(t) &\geq \log\paren{\sum_{j=0}^{T-1}\sum_{n=1}^{\infty} \frac{1}{\paren{\abs{e(j/T)d_{n}}+1}^{2t}}} = \log\paren{T\sum_{n=1}^{\infty} \frac{1}{(d_{n}+1)^{2t}}}
    \end{align*}
    for any $t \geq 0$.
\end{proof}

\begin{thm}\label{thm:comparison_of_f_and_g}
    Let $\mathbf{d} = \set{d_n}_{n \in \mathbb{N}} \in \mathcal{D}$, $T \in \mathbb{N}$.
    Suppose that $\mathcal{F}\paren{\mathbf{d}, T}$ and $\mathcal{G}\paren{\mathbf{d}, T}$ are CIFSs.
    Then, the following (1) and (2) hold.
    \begin{enumerate}
        \item $F\paren{\mathcal{F}\paren{\mathbf{d}, T}} = F\paren{\mathcal{G}\paren{\mathbf{d}, T}}$. In particular, $\theta_{\mathcal{F}\paren{\mathbf{d}, T}} = \theta_{\mathcal{G}\paren{\mathbf{d}, T}}$.
        \item $\mathcal{F}\paren{\mathbf{d}, T}$ is hereditarily regular if and only if $\mathcal{G}\paren{\mathbf{d}, T}$ is hereditarily regular.
    \end{enumerate}
\end{thm}

\begin{proof}
    Fix $t \geq 0$. The three series
    \begin{equation*}
        \sum_{n=1}^{\infty} \frac{1}{(d_{n}+1)^{2t}}, \quad \sum_{n=1}^{\infty} \frac{1}{(d_{n}-1)^{2t}}, \quad \sum_{n=1}^{\infty} \frac{1}{(d_{n}^{2}-1)^{t}}
    \end{equation*}
    either converge or diverge simultaneously by the comparison test.
    Therefore, (1) follows from \Cref{lem:pressure_of_f} and \Cref{lem:pressure_ineq_g}.
    (2) also follows from (1), \Cref{lem:pressure_of_f} and \Cref{lem:pressure_ineq_g}.
\end{proof}

The results of Dirichlet series in analytic number theory
help us calculate the values of $\theta_{\mathcal{F}\paren{\mathbf{d}, T}}$ and $\theta_{\mathcal{G}\paren{\mathbf{d}, T}}$.
We now quote these results from \cite{Zagier-Katayama-1990-en}.

\begin{definition}
    Let $\{a_{n}\}_{n\in\mathbb{N}}$ be a sequence of complex numbers and
    let $\{\lambda_{n}\}_{n\in\mathbb{N}}$ be a strictly increasing sequence of real numbers
    satisfying $\lambda_{n} \to \infty \ (n\to\infty)$.
    We call the series defined by
    \begin{equation*}
        \sum_{n=1}^{\infty} a_{n} e^{-\lambda_{n}s} \quad (s\in\mathbb{C})
    \end{equation*}
    a Dirichlet series.
\end{definition}

\begin{thm}[\cite{Zagier-Katayama-1990-en} Part I §1 Theorem 1, Theorem 2]\label{thm:Dirichlet_series}
    Let $\sum_{n=1}^{\infty} a_{n} e^{-\lambda_{n}s}$ be a Dirichlet series.
    Then, the following (1) and (2) hold.
    \begin{enumerate}
        \item There exists $\sigma \in [-\infty, \infty]$ such that
            the Dirichlet series $\sum_{n=1}^{\infty} a_{n} e^{-\lambda_{n}s}$
            converges if $\mathrm{Re}(s) > \sigma$ and diverges if $\mathrm{Re}(s) < \sigma$.
        \item Suppose that $\sum_{n=1}^{\infty}a_{n}$ diverges.
            Then, the $\sigma$ in (1) is given by
        \begin{equation*}
            \sigma = \limsup_{N\to\infty} \frac{\log\left|\sum_{n=1}^{N} a_{n}\right|}{\lambda_{N}}.
        \end{equation*}
    \end{enumerate}
\end{thm}

This theorem allows us to calculate $\theta_{\mathcal{F}\paren{\mathbf{d}, T}}$ and $\theta_{\mathcal{G}\paren{\mathbf{d}, T}}$ as follows.

\begin{thm}\label{thm:theta_via_Dirichlet}
    Let $\mathcal{F}\paren{\mathbf{d}, T}$ be a CIFS.
    Then,
    \begin{equation*}
        \theta_{\mathcal{F}\paren{\mathbf{d}, T}} = \limsup_{N\to\infty} \frac{\log{N}}{2\log{d_{N}}}.
    \end{equation*}
    Also, if $\mathcal{G}\paren{\mathbf{d}, T}$ is a CIFS,
    \begin{equation*}
        \theta_{\mathcal{G}\paren{\mathbf{d}, T}} = \limsup_{N\to\infty} \frac{\log{N}}{2\log{d_{N}}}.
    \end{equation*}
\end{thm}

\begin{proof}
    Since
    \begin{equation*}
        \sum_{n=1}^{\infty} \frac{1}{(d_{n}^{2}-1)^{t}} = \sum_{n=1}^{\infty}1\cdot e^{-t\log{(d_{n}^{2}-1)}},
    \end{equation*}
    by \Cref{thm:Dirichlet_series}(2),
    the definition of $\theta_{\mathcal{F}\paren{\mathbf{d}, T}}$, \Cref{lem:pressure_of_f} and $d_{n} \to \infty \ (n\to\infty)$,
    we have
    \begin{equation*}
        \theta_{\mathcal{F}\paren{\mathbf{d}, T}} = \limsup_{N\to\infty} \frac{\log{N}}{\log{(d_{N}^{2}-1)}} = \limsup_{N\to\infty} \frac{\log{N}}{2\log{d_{N}}}.
    \end{equation*}
    The same argument works for $\theta_{\mathcal{G}\paren{\mathbf{d}, T}}$.
\end{proof}

\subsection{The Hausdorff and Packing Measures of \texorpdfstring{$J_{\mathcal{F}}$ and $J_{\mathcal{G}}$}{J\_F and J\_G}}
Throughout this subsection, we assume that
$\mathcal{F} = \mathcal{F}(\mathbf{d}, T)$ and $\mathcal{G} = \mathcal{G}(\mathbf{d}, T)$
are regular CIFSs for some
$\mathbf{d} = \set{d_{n}}_{n=1}^{\infty} \in \mathcal{D}$ and $T \in \mathbb{N}$.

\begin{definition}
    For $0 < r < \frac{1}{d_{1}+1}$, we define $n(r)$ by
    \begin{equation*}
        n(r) := \max\left\{n\in\mathbb{N} \mid \frac{1}{d_{n}+1} \geq r\right\}.
    \end{equation*}
\end{definition}

\begin{lem}\label{lem:image_estimate}
    Let $r>0$.
    Then, the following statements hold.
    \begin{enumerate}
        \item If $n\geq n(r) + 2$, then for $j = 0, 1, \dots, T-1$,
            we have
            \begin{equation*}
                B(0, r) \supset f_{e\paren{j/T}d_{n}}\paren{\overline{\mathbb{D}}} = g_{e\paren{j/T}d_{n}}\paren{\overline{\mathbb{D}}}.
            \end{equation*}
        \item If $n\leq n(r)$, then for $j = 0, 1, \dots, T-1$,
            we have
            \begin{equation*}
                B(0, r) \cap f_{e\paren{j/T}d_{n}}\paren{\overline{\mathbb{D}}}= B(0, r) \cap g_{e\paren{j/T}d_{n}}\paren{\overline{\mathbb{D}}} = \emptyset.
            \end{equation*}
    \end{enumerate}
\end{lem}

\begin{proof}
    Since
    \begin{equation*}
        \frac{1}{d_{n(r)+2}-1} \leq \frac{1}{d_{n(r)+1}+1} < r \leq \frac{1}{d_{n(r)}+1}
    \end{equation*}
    and
    \begin{equation*}
        f_{d_{n}}\paren{\overline{\mathbb{D}}} = g_{d_{n}}\paren{\overline{\mathbb{D}}} = \overline{B}\paren{\frac{d_{n}}{d_{n}^{2}-1}, \frac{1}{d_{n}^{2}-1}},
    \end{equation*}
    (1) and (2) hold if $j = 0$. The other cases are similar to the above.
\end{proof}

\begin{lem}\label{lem:measure_estimate}
    Let $r>0$. Then, the following inequalities hold.
    \begin{align*}
        T\sum_{n=n(r)+2}^{\infty} \frac{1}{(d_{n}^{2}-1)^{h_{\mathcal{F}}}} \leq m_{\mathcal{F}}(B(0, r)) \leq T\sum_{n=n(r)+1}^{\infty} \frac{1}{(d_{n}^{2}-1)^{h_{\mathcal{F}}}}, \\
        T\sum_{n=n(r)+2}^{\infty} \frac{1}{(d_{n}+1)^{2h_{\mathcal{G}}}} \leq m_{\mathcal{G}}(B(0, r)) \leq T\sum_{n=n(r)+1}^{\infty} \frac{1}{(d_{n}-1)^{2h_{\mathcal{G}}}}.
    \end{align*}
\end{lem}

\begin{proof}
    We prove the inequalities for $m_{\mathcal{F}}(B(0, r))$.
    Let $r>0$. By \Cref{lem:image_estimate}, we have
    \begin{align*}
        m_{\mathcal{F}}(B(0, r)) &= m_{\mathcal{F}}\paren{B(0, r) \cap J_{\mathcal{F}}} = m_{\mathcal{F}}\paren{B(0, r) \cap \bigcup_{j=0}^{T-1}\bigcup_{n \geq 1} f_{e\paren{j/T}d_{n}}\paren{\overline{\mathbb{D}}}} \\
        &\geq m_{\mathcal{F}}\paren{B(0, r) \cap \bigcup_{j=0}^{T-1}\bigcup_{n \geq n(r)+2} f_{e\paren{j/T}d_{n}}\paren{\overline{\mathbb{D}}}} \\
        &\geq \sum_{j=0}^{T-1}\sum_{n = n(r)+2}^{\infty} m_{\mathcal{F}}\paren{f_{e\paren{j/T}d_{n}}\paren{\overline{\mathbb{D}}}} \\
        &= T\sum_{n = n(r)+2}^{\infty} \frac{1}{(d_{n}^{2}-1)^{h_{\mathcal{F}}}}.
    \end{align*}
    Here, the last equality follows from the $h_{\mathcal{F}}$-conformality of $m_{\mathcal{F}}$.
    The other inequality can be shown similarly. The same argument applies to the inequalities for $m_{\mathcal{G}}(B(0, r))$.
\end{proof}

For general $\mathbf{d} \in \mathcal{D}$ and $T \in \mathbb{N}$,
we have the following criteria.

\begin{thm}[\Cref{main1}]\label{thm:density-at-zero-and-packing-measure-infinity-main1}
    \begin{enumerate}
        \item Suppose $P_{\mathcal{F}}\paren{h_{\mathcal{F}}/2} < \infty$. Then, we have
        \begin{equation*}
            \limsup_{r \downarrow 0}\frac{m_{\mathcal{F}}(B(0,r))}{r^{h_{\mathcal{F}}}} = 0, \quad \Pi_{h_{\mathcal{F}}}(J_{\mathcal{F}}) = \infty.
        \end{equation*}

        \item Suppose $P_{\mathcal{G}}\paren{h_{\mathcal{G}}/2} < \infty$. Then, we have
        \begin{equation*}
            \limsup_{r \downarrow 0}\frac{m_{\mathcal{G}}(B(0,r))}{r^{h_{\mathcal{G}}}} = 0, \quad \Pi_{h_{\mathcal{G}}}(J_{\mathcal{G}}) = \infty.
        \end{equation*}
    \end{enumerate}
\end{thm}

\begin{proof}
    We prove (1). Let $r>0$. By \Cref{lem:measure_estimate}, we have
    \begin{align*}
        \frac{m_{\mathcal{F}}(B(0, r))}{r^{h_{\mathcal{F}}}} &\leq \frac{T\sum_{n = n(r)+1}^{\infty} \frac{1}{(d_{n}^{2}-1)^{h_{\mathcal{F}}}}}{\paren{\frac{1}{d_{n(r)+1}+1}}^{h_{\mathcal{F}}}} = T \sum_{n = n(r)+1}^{\infty} \paren{\frac{d_{n(r)+1}+1}{d_{n}^{2}-1}}^{h_{\mathcal{F}}} \\
        &= T\sum_{n = n(r)+1}^{\infty}\paren{\frac{d_{n(r)+1}+1}{d_{n}+1}\frac{d_{n}+1}{d_{n}^{2}-1}}^{h_{\mathcal{F}}} \\
        &\leq T\sum_{n = n(r)+1}^{\infty}\paren{\frac{1}{(d_{n}-1)^{2}}}^{h_{\mathcal{F}}/2}.
    \end{align*}
    Note that $P_{\mathcal{F}}(h_{\mathcal{F}}/2) < \infty$ holds by the assumption.
    Therefore, by the comparison test and \Cref{lem:pressure_of_f}, we obtain
    \begin{equation*}
        T\sum_{n = 1}^{\infty}\paren{\frac{1}{(d_{n}-1)^{2}}}^{h_{\mathcal{F}}/2} < \infty.
    \end{equation*}
    Since $n(r) \to \infty$ as $r \downarrow 0$,
    \begin{equation*}
        \limsup_{r\downarrow 0} \frac{m_{\mathcal{F}}(B(0, r))}{r^{h_{\mathcal{F}}}} \leq \limsup_{r\downarrow 0}\paren{T\sum_{n = n(r)+1}^{\infty}\paren{\frac{1}{(d_{n}-1)^{2}}}^{h_{\mathcal{F}}/2}} = 0.
    \end{equation*}
    Thus, it follows from \Cref{thm:packing-measure-x-infty} that
    $\Pi_{h_{\mathcal{F}}}(J_{\mathcal{F}}) = \infty$.
    We complete the proof of (1).

    (2) can be shown similarly by using \Cref{lem:measure_estimate,lem:pressure_ineq_g}
    and \Cref{cor:hausdorff-measure-x-infty}.
\end{proof}

The following two lemmas are direct consequences of \Cref{thm:density-at-zero-and-packing-measure-infinity-main1,thm:hereditarily-regularity-and-pressure}.

\begin{cor}\label{cor:packing_infty_when_hereditarily_regular}
    \begin{enumerate}
        \item Suppose $\mathcal{F}$ is hereditarily regular.
            If $h_{\mathcal{F}} > 2\theta_{\mathcal{F}}$, then
            \begin{equation*}
                \limsup_{r \downarrow 0}\frac{m_{\mathcal{F}}(B(0,r))}{r^{h_{\mathcal{F}}}} = 0, \quad \Pi_{h_{\mathcal{F}}}(J_{\mathcal{F}}) = \infty.
            \end{equation*}

        \item Suppose $\mathcal{G}$ is hereditarily regular.
            If $h_{\mathcal{G}} > 2\theta_{\mathcal{G}}$, then
            \begin{equation*}
                \limsup_{r \downarrow 0}\frac{m_{\mathcal{G}}(B(0,r))}{r^{h_{\mathcal{G}}}} = 0, \quad \Pi_{h_{\mathcal{G}}}(J_{\mathcal{G}}) = \infty.
            \end{equation*}
    \end{enumerate}
\end{cor}

\begin{cor}
    \begin{enumerate}
        \item Suppose $\mathcal{F}$ is regular but not hereditarily regular.
            If $h_{\mathcal{F}} \geq 2\theta_{\mathcal{F}}$, then
            \begin{equation*}
                \limsup_{r \downarrow 0}\frac{m_{\mathcal{F}}(B(0,r))}{r^{h_{\mathcal{F}}}} = 0, \quad \Pi_{h_{\mathcal{F}}}(J_{\mathcal{F}}) = \infty.
            \end{equation*}

        \item Suppose $\mathcal{G}$ is regular but not hereditarily regular.
            If $h_{\mathcal{G}} \geq 2\theta_{\mathcal{G}}$, then
            \begin{equation*}
                \limsup_{r \downarrow 0}\frac{m_{\mathcal{G}}(B(0,r))}{r^{h_{\mathcal{G}}}} = 0, \quad \Pi_{h_{\mathcal{G}}}(J_{\mathcal{G}}) = \infty.
            \end{equation*}
    \end{enumerate}
\end{cor}

Let us consider an infinite subset $I$ of $\mathbb{N}$
with $\min I > 2$ and $\abs{i-j} \geq 2$ for any distinct $i, j \in I$.
We can regard $I$ as a sequence belonging to $\mathcal{D}$.
Then, as for $\widetilde{\mathcal{G}}$  defined in \Cref{sec:introduction},
we have $J_{\mathcal{G}(I, 1)} = J_{\widetilde{\mathcal{G}}(I, 1)}$.
Therefore, \Cref{cor:packing_infty_when_hereditarily_regular} is a
partial generalization of \Cref{thm:MauldinUrbanski1999-4.4-5.2}(2).
Thus, we conjecture the following.

\begin{conjecture}\label{conj:hausdorff_measure_zero_when_less_than_2theta}
    \begin{enumerate}
        \item Suppose $\mathcal{F}$ is a regular CIFS.
            If $h_{\mathcal{F}} < 2\theta_{\mathcal{F}}$, then
            $H_{h_{\mathcal{F}}}(J_{\mathcal{F}}) =0$.
        \item Suppose $\mathcal{G}$ is a regular CIFS.
            If $h_{\mathcal{G}} < 2\theta_{\mathcal{G}}$, then
            $H_{h_{\mathcal{G}}}(J_{\mathcal{G}}) =0$.
    \end{enumerate}
\end{conjecture}

We prove that this conjecture is true in the case where
$d_{n}$ grows polynomially with respect to $n$
in \Cref{thm:measure_poly_f,thm:measure_poly_g}.
However, the general case remains open.
Note that the digits are restricted to natural numbers in \Cref{thm:MauldinUrbanski1999-4.4-5.2}
and that Mauldin and Urba\'nski did not use the density theorems like \Cref{thm:hausdorff-measure-x-infty}
to show \Cref{thm:MauldinUrbanski1999-4.4-5.2}.

\section{Relationship between the growth rate of digits and dimensions/measures}\label{sec:relationship_between_increasing_rate_and_measures}
In this section,
we investigate the relationship between the growth rate of the sequence $\set{d_{n}}_{n=1}^{\infty}$
and the properties of the CIFSs $\mathcal{F}$ and $\mathcal{G}$.
In particular, we focus on the case of polynomial growth.
Throughout this section, we always assume that $\mathbf{d} = \set{d_n}_{n=1}^{\infty} \in \mathcal{D}$
satisfies the following conditions:
there exist $0 < c_{1} \leq c_{2}, c_{3} > 0, \gamma \geq 1$ and $N \in \mathbb{N}$ such that
\begin{equation}\label{eq:polynomial_growth_conditions}
    c_{1}n^{\gamma} \leq d_{n} \leq c_{2}n^{\gamma}, \quad d_{n+1} - d_{n} \geq c_{3}n^{\gamma-1}, \quad c_{1}N^{\gamma} > 1 \quad (N \leq n \in \mathbb{N}).
\end{equation}
Note that if $c_{1} = c_{2} =: c$, then
\begin{equation*}
    d_{n+1} - d_{n} = cn^{\gamma}\paren{\paren{1+\frac{1}{n}}^{\gamma} - 1} \geq cn^{\gamma}\paren{1 + \frac{\gamma}{n} - 1} = c\gamma n^{\gamma - 1}.
\end{equation*}
Thus, the second condition in \eqref{eq:polynomial_growth_conditions} holds
when the first condition holds with $c_{1} = c_{2} =: c$ and $c_{3} = c\gamma$.

Throughout this section, we suppose that both $\mathcal{F} = \mathcal{F}(\mathbf{d}, T)$ and $\mathcal{G} = \mathcal{G}(\mathbf{d}, T)$
are CIFSs for $\mathbf{d} \in \mathcal{D}$ satisfying \eqref{eq:polynomial_growth_conditions} and $T \in \mathbb{N}$.

\subsection{Zero Hausdorff Measure and Infinite Packing Measure}

We first state the properties of $\theta_{\mathcal{F}}$ and $\theta_{\mathcal{G}}$.
\begin{prop}\label{prop:theta_poly}
    We have $\theta_{\mathcal{F}} = \theta_{\mathcal{G}} = \frac{1}{2\gamma}$.
\end{prop}

\begin{proof}
    This follows immediately from \Cref{thm:theta_via_Dirichlet} and the assumption \eqref{eq:polynomial_growth_conditions}.
\end{proof}

\begin{prop}\label{prop:hereditarily_regular_poly}
    Both of $\mathcal{F}$ and $\mathcal{G}$ are hereditarily regular.
\end{prop}

\begin{proof}
    Since $d_{n} \leq c_{2}n^{\gamma}$ by \eqref{eq:polynomial_growth_conditions},
    \Cref{prop:theta_poly} and \Cref{lem:pressure_of_f} imply
    \begin{equation*}
        P_{\mathcal{F}}(\theta_{\mathcal{F}}) = \log\paren{T\sum_{n=1}^{\infty} \frac{1}{(d_{n}^{2}-1)^{\theta_{\mathcal{F}}}}} \geq \log\paren{T\sum_{n=1}^{\infty} \frac{1}{(c_{2}^{2}n^{2\gamma}-1)^{\frac{1}{2\gamma}}}} = \infty.
    \end{equation*}
    Therefore, $\mathcal{F}$ is hereditarily regular by \Cref{thm:hereditarily-regularity-and-pressure}.
    The same argument works for $\mathcal{G}$ by \Cref{lem:pressure_ineq_g}.
\end{proof}

We now state the properties related to (C1) and (C2),
which are peculiar to infinite IFSs.
The following two theorems show that
\Cref{conj:hausdorff_measure_zero_when_less_than_2theta} is true
under the assumption \eqref{eq:polynomial_growth_conditions}.

\begin{thm}[\Cref{main2}-(1)-(a)]\label{thm:measure_poly_f}
    We have the following.
    \begin{enumerate}
        \item If $h_{\mathcal{F}} < \frac{1}{\gamma}$, then $\limsup_{r\downarrow 0} \frac{m_{\mathcal{F}}(B(0, r))}{r^{h_{\mathcal{F}}}} = \infty$ and $H_{h_{\mathcal{F}}}(J_{\mathcal{F}}) = 0$.
        \item If $h_{\mathcal{F}} > \frac{1}{\gamma}$, then $\limsup_{r\downarrow 0} \frac{m_{\mathcal{F}}(B(0, r))}{r^{h_{\mathcal{F}}}} = 0$ and $\Pi_{h_{\mathcal{F}}}(J_{\mathcal{F}}) = \infty$.
    \end{enumerate}
\end{thm}

\begin{proof}
    (2) follows from \Cref{prop:hereditarily_regular_poly} and \Cref{cor:packing_infty_when_hereditarily_regular}.
    We now prove (1). Let $r>0$ be small enough such that $n(r) > N$.
    Since $h_{\mathcal{F}} > \theta_{\mathcal{F}} = \frac{1}{2\gamma}$
    by \Cref{prop:theta_poly,prop:hereditarily_regular_poly} and \Cref{thm:theta-and-dimh-if-hereditarily-regular},
    we obtain
    \begin{align}
        \sum_{k=n(r)+2}^{\infty}\frac{1}{(d_{k}^{2}-1)^{h_{\mathcal{F}}}} &\geq \sum_{k=n(r)+2}^{\infty} \frac{1}{(c_{2}^{2}k^{2\gamma}-1)^{h_{\mathcal{F}}}} = \frac{1}{c_{2}^{2h_{\mathcal{F}}}} \sum_{k=n(r)+2}^{\infty} \frac{1}{k^{2\gamma h_{\mathcal{F}}}(1-k^{-2\gamma}c_{2}^{-2})^{h_{\mathcal{F}}}} \notag \\
        &\geq \frac{1}{c_{2}^{2h_{\mathcal{F}}}} \sum_{k=n(r)+2}^{\infty} \frac{1}{k^{2\gamma h_{\mathcal{F}}}} \notag \\
        &\geq \frac{1}{c_{2}^{2h_{\mathcal{F}}}} \int_{n(r)+2}^{\infty} \frac{1}{x^{2\gamma h_{\mathcal{F}}}} dx \notag \\
        &= \const \cdot (n(r)+2)^{-2\gamma h_{\mathcal{F}} + 1} \label{eq:measure_poly_f_proof_1},
    \end{align}
    where $\const$ denotes a constant independent of $r$.
    We also have
    \begin{equation*}
        r \leq \frac{1}{d_{n(r)}+1} \leq \frac{1}{d_{n(r)}} \leq \frac{1}{c_{1}n(r)^{\gamma}}
    \end{equation*}
    by the definition of $n(r)$.
    Therefore, it follows from \Cref{lem:measure_estimate} and \eqref{eq:measure_poly_f_proof_1} that
    \begin{align*}
        \frac{m_{\mathcal{F}}(B(0, r))}{r^{h_{\mathcal{F}}}} &\geq \const \cdot (n(r)+2)^{-2\gamma h_{\mathcal{F}} + 1}(c_{1}n(r)^{\gamma})^{h_{\mathcal{F}}} \\
        &= \const \cdot n(r)^{-\gamma h_{\mathcal{F}} + 1} \paren{\frac{n(r)+2}{n(r)}}^{-2\gamma h_{\mathcal{F}} + 1}.
    \end{align*}
    Thus, if $h_{\mathcal{F}} < \frac{1}{\gamma}$, then
    \begin{equation*}
        \limsup_{r\downarrow 0} \frac{m_{\mathcal{F}}(B(0, r))}{r^{h_{\mathcal{F}}}} = \infty.
    \end{equation*}
    Hence, by \Cref{thm:hausdorff-measure-x-infty}, we have $H_{h_{\mathcal{F}}}(J_{\mathcal{F}}) = 0$.
    We complete the proof of our theorem.
\end{proof}

\begin{thm}[\Cref{main2}-(2)-(a)]\label{thm:measure_poly_g}
    We have the following.
    \begin{enumerate}
        \item If $h_{\mathcal{G}} < \frac{1}{\gamma}$, then $\limsup_{r\downarrow 0} \frac{m_{\mathcal{G}}(B(0, r))}{r^{h_{\mathcal{G}}}} = \infty$ and $H_{h_{\mathcal{G}}}(J_{\mathcal{G}}) = 0$.
        \item If $h_{\mathcal{G}} > \frac{1}{\gamma}$, then $\limsup_{r\downarrow 0} \frac{m_{\mathcal{G}}(B(0, r))}{r^{h_{\mathcal{G}}}} = 0$ and $\Pi_{h_{\mathcal{G}}}(J_{\mathcal{G}}) = \infty$.
    \end{enumerate}
\end{thm}

\begin{proof}
    The proof is similar to that of \Cref{thm:measure_poly_f}.
\end{proof}

We continue to study the remaining parts of \Cref{main2}.

\begin{lem}\label{lem:box-dimension-of-g-poly}
    For all $x\in [0, 1]$, we have $\underline{\dim}_{B}(\mathcal{G}_{x}) = \overline{\dim}_{B}(\mathcal{G}_{x}) = \frac{1}{\gamma+1}$.
    Here, as defined in \Cref{def:S_x},
    \begin{equation*}
        \mathcal{G}_{x} = \set{\frac{1}{e(j/T)d_{n}+x} \mid j = 0, 1, \dots, T-1; n\in\mathbb{N}}.
    \end{equation*}
\end{lem}

\begin{proof}
    Let $x \in [0, 1]$.
    For each $r > 0$, set $k(r)$ be the minimum number of closed balls with radius $r$
    that cover $\mathcal{G}_{x}$. Define $r_{0}$ by
    \begin{equation*}
        r_{0} := \min_{1 \leq n \leq N}\frac{1}{2}\paren{\frac{1}{d_{n}+x} - \frac{1}{d_{n+1}+x}}
    \end{equation*}
    and take $r \in (0, r_{0})$.
    If $1 \leq n \leq N$, then
    \begin{equation*}
        \frac{1}{d_{n}+x} - \frac{1}{d_{n+1}+x} \geq 2 r_{0} > 2 r.
    \end{equation*}
    In addition, if $N < n < N_{1}(r)$, where
    \begin{equation}
        N_{1}(r) := \paren{2c_{2}^{2}c_{3}^{-1}\paren{1+\frac{1}{c_{2}N^{\gamma}}}^{2}r}^{-1/(\gamma+1)}-1 \quad (0 < r < r_{0}),
    \end{equation}
    then, by \eqref{eq:polynomial_growth_conditions}, we have
    \begin{align*}
        \frac{1}{d_{n}+x} - \frac{1}{d_{n+1}+x} & = \frac{d_{n+1} - d_{n}}{(d_{n}+x)(d_{n+1}+x)} \geq \frac{c_{3}n^{\gamma-1}}{\paren{c_{2}n^{\gamma}+x}\paren{c_{2}(n+1)^{\gamma}+x}}\\
        &= \frac{c_{3}n^{\gamma-1}}{c_{2}^{2}n^{\gamma}\paren{n+1}^{\gamma}{\paren{1+\frac{x}{c_{2}n^{\gamma}}}\paren{1+\frac{x}{c_{2}(n+1)^{\gamma}}}}} \\
        &\geq \frac{c_{3}}{c_{2}^{2}(n+1)^{\gamma+1}\paren{1+\frac{1}{c_{2}N^{\gamma}}}^{2}} > 2 r.
    \end{align*}
    Therefore, any closed ball with radius $r$ cannot contain both of $1/(d_{n}+x)$ and $1/(d_{n+1}+x)$
    for $1 \leq n < N_{1}(r)$.
    Thus, if $r$ is sufficiently small, then we have
    \begin{equation*}
        k(r) \geq N_{1}(r).
    \end{equation*}
    Hence, we obtain
    \begin{equation}\label{eq:box-dim-lower-bound}
        \underline{\dim}_{B}(\mathcal{G}_{x}) = \liminf_{r\downarrow 0} \frac{\log k(r)}{-\log r} \geq \frac{1}{\gamma+1}.
    \end{equation}

    We set
    \begin{equation*}
        N_{2}(r) := \left\lceil \paren{\frac{1}{c_{1}}\cdot\paren{\frac{2c_{1}r}{\gamma}}^{-\frac{\gamma}{\gamma+1}} - x}^{1/\gamma} \right\rceil.
    \end{equation*}
    If we take $r > 0$ small enough to satisfy $N_{2}(r) > N$,
    then,
    \begin{equation*}
        \left[0, \frac{1}{d_{N_{2}(r)}+x}\right] \subset \left[0, \frac{1}{c_{1}N_{2}(r)^{\gamma}+x}\right] \subset \left[0, \paren{\frac{2c_{1}r}{\gamma}}^{\frac{\gamma}{\gamma+1}}\right].
    \end{equation*}
    Therefore, we have
    \begin{equation*}
        k(r) \leq T\left\{\paren{\frac{2c_{1}r}{\gamma}}^{\frac{\gamma}{\gamma+1}} \cdot \frac{1}{2r} + N_{2}(r)\right\} \leq T\paren{\const \cdot r^{-\frac{1}{\gamma+1}}},
    \end{equation*}
    where $\const$ is a constant independent of $r$.
    Thus, we obtain
    \begin{equation}\label{eq:box-dim-upper-bound}
        \overline{\dim}_{B}(\mathcal{G}_{x}) = \limsup_{r\downarrow 0} \frac{\log k(r)}{-\log r} \leq \frac{1}{\gamma+1}.
    \end{equation}

    It follows from \eqref{eq:box-dim-lower-bound} and \eqref{eq:box-dim-upper-bound} that
    \begin{equation*}
        \underline{\dim}_{B}(\mathcal{G}_{x}) = \overline{\dim}_{B}(\mathcal{G}_{x}) = \frac{1}{\gamma+1}.
    \end{equation*}
    Thus, we have proved our lemma.
\end{proof}

\begin{lem}\label{lem:box-dimension-of-f-poly}
    For all $x\in [0, 1]$, we have
    $\underline{\dim}_{B}(\mathcal{F}_{x}) = \overline{\dim}_{B}(\mathcal{F}_{x}) = \frac{1}{\gamma+1}$.
\end{lem}

\begin{proof}
    Let $x \in [0, 1]$. For each $n \in \mathbb{N}$ with $n \geq N$, we have
    \begin{align*}
        \frac{d_{n}+x}{d_{n}^{2}-1} - \frac{d_{n+1}+x}{d_{n+1}^{2}-1} &= \frac{1 + x/d_{n}}{d_{n} - 1/d_{n}} - \frac{1 + x/d_{n+1}}{d_{n+1} - 1/d_{n+1}} \\
        &= \frac{(1 + x/d_{n})(d_{n+1}-1/d_{n+1}) - (1 + x/d_{n+1})(d_{n} - 1/d_{n})}{(d_{n} - 1/d_{n})(d_{n+1} - 1/d_{n+1})} \\
        &\geq \frac{d_{n+1} - d_{n} + x\paren{d_{n+1}/d_{n} - d_{n}/d_{n+1}} + 1/d_{n} - 1/d_{n+1}}{d_{n}d_{n+1}} \\
        &\geq \frac{d_{n+1} - d_{n}}{d_{n}d_{n+1}} \geq \frac{c_{3}n^{\gamma-1}}{c_{2}^{2}n^{\gamma}(n+1)^{\gamma}} \\
        &\geq \frac{c_{3}}{c_{2}^{2}} (n+1)^{-\gamma - 1}.
    \end{align*}
    We also have
    \begin{equation*}
        \bracket{0, \frac{d_{n}+x}{d_{n}^{2}-1}} \subset \bracket{0, \frac{c_{2}n^{\gamma}+1}{\paren{c_{1}n^{\gamma}}^{2}- 1}} \subset \bracket{0, \frac{c_{2} + \frac{1}{N^{\gamma}}}{c_{1}^{2}n^{\gamma} - \frac{1}{c_{1}N^{\gamma}}}}
    \end{equation*}
    for $n \geq N$.
    Therefore, by setting
    \begin{equation*}
        N_{1}(r) := \paren{\frac{2c_{2}^{2}r}{c_{3}}}^{-1/(\gamma+1)} - 1,  \quad N_{2}(r) := \left\lceil \paren{c_{1}^{-1}\paren{c_{2} + N^{-\gamma}}r^{-\gamma/(\gamma+1)} + N^{\gamma}}^{1/\gamma} \right\rceil,
    \end{equation*}
    we can show $\underline{\dim}_{B}(\mathcal{F}_{x}) = \overline{\dim}_{B}(\mathcal{F}_{x}) = \frac{1}{\gamma+1}$
    by the same argument as in the proof of \Cref{lem:box-dimension-of-g-poly}.
\end{proof}

\begin{cor}\label{cor:dimh-dimp-packing-infty}
    \begin{enumerate}
        \item If $h_{\mathcal{F}} \geq \frac{1}{\gamma+1}$, then $\dimp{J_{\mathcal{F}}} = h_{\mathcal{F}}$.
        \item If $h_{\mathcal{G}} \geq \frac{1}{\gamma+1}$, then $\dimp{J_{\mathcal{G}}} = h_{\mathcal{G}}$.
    \end{enumerate}
\end{cor}

\begin{proof}
    These are the direct consequences of \Cref{lem:box-dimension-of-f-poly,lem:box-dimension-of-g-poly} and \Cref{thm:packing-dimension-and-box-dimension}.
\end{proof}

\begin{cor}[\Cref{main2}-(b)]\label{cor:hausdorff-dim-equals-packing-dim-and-packing-measure-infty}
    \begin{enumerate}
        \item If $h_{\mathcal{F}} > \frac{1}{\gamma}$, then $\dimp{J_{\mathcal{F}}} = h_{\mathcal{F}}$ and $\Pi_{h_{\mathcal{F}}}(J_{\mathcal{F}}) = \infty$.
        \item If $h_{\mathcal{G}} > \frac{1}{\gamma}$, then $\dimp{J_{\mathcal{G}}} = h_{\mathcal{G}}$ and $\Pi_{h_{\mathcal{G}}}(J_{\mathcal{G}}) = \infty$.
    \end{enumerate}
\end{cor}

\begin{proof}
    The statements follow from Theorems \ref{thm:measure_poly_f}(2), \ref{thm:measure_poly_g}(2) and \Cref{cor:dimh-dimp-packing-infty}.
\end{proof}

\begin{cor}\label{cor:lower-box-dim-poly}
    We have the following lower bounds for the lower box dimensions of $J_{\mathcal{F}}$ and $J_{\mathcal{G}}$.
    \begin{enumerate}
        \item $\dimlb{J_{\mathcal{F}}} \geq \frac{1}{\gamma+1}$.
        \item $\dimlb{J_{\mathcal{G}}} \geq \frac{1}{\gamma+1}$.
    \end{enumerate}
\end{cor}

\begin{proof}
    We prove (1). Note that $J_{\mathcal{F}} \cap [0, 1] \ne \emptyset$ and
    fix $x \in J_{\mathcal{F}} \cap [0, 1]$.
    Since $\mathcal{F}_{x} \subset J_{\mathcal{F}}$ by \Cref{prop:self-similarity-of-limit-set},
    it follows from \Cref{lem:box-dimension-of-f-poly} that
    \begin{equation*}
        \dimlb{J_{\mathcal{F}}} \geq \underline{\dim}_{B}(\mathcal{F}_{x}) = \frac{1}{\gamma+1}.
    \end{equation*}
    The same argument works for (2) by using \Cref{lem:box-dimension-of-g-poly}.
\end{proof}

The following theorem is an analogue of \Cref{thm:MauldinUrbanski1999-6.1}.
\begin{thm}[\Cref{main2}-(c)]\label{thm:measure_poly_gamma_general_T_large}
    Suppose $T \geq c_{2}^{2/\gamma}\cdot (N+1)$. Then,
    \begin{equation*}
        h_{\mathcal{F}} > \frac{1}{\gamma}, \quad h_{\mathcal{G}} > \frac{1}{\gamma}.
    \end{equation*}
    Therefore, in particular, by \Cref{cor:hausdorff-dim-equals-packing-dim-and-packing-measure-infty},
    \begin{equation*}
        \dimp{J_{\mathcal{F}}} = h_{\mathcal{F}},\quad  \dimp{J_{\mathcal{G}}} = h_{\mathcal{G}}, \quad \Pi_{h_{\mathcal{F}}}(J_{\mathcal{F}}) = \Pi_{h_{\mathcal{G}}}(J_{\mathcal{G}}) = \infty.
    \end{equation*}
\end{thm}

\begin{proof}
    We prove the statement for $\mathcal{F}$.
    By \Cref{lem:pressure_of_f} and $d_{n}^{2} - 1 \leq d_{n+1}^{2}$, we have
    \begin{equation*}
        P_{\mathcal{F}}\paren{\frac{1}{\gamma}} = \log\paren{T\sum_{n=1}^{\infty} \frac{1}{(d_{n}^{2}-1)^{\frac{1}{\gamma}}}} \geq \log\paren{T\sum_{n=1}^{\infty} \frac{1}{d_{n+1}^{2/\gamma}}} \geq \log\paren{\frac{T}{c_{2}^{2/\gamma}} \sum_{n=N+1}^{\infty} \frac{1}{n^{2}}}.
    \end{equation*}
    Since $\sum_{n=N+1}^{\infty} \frac{1}{n^{2}} > \int_{N+1}^{\infty} \frac{1}{x^{2}} dx = \frac{1}{N+1}$,
    we obtain
    \begin{equation*}
        P_{\mathcal{F}}\paren{\frac{1}{\gamma}} > \log\paren{\frac{T}{c_{2}^{2/\gamma}\cdot (N+1)}} \geq 0.
    \end{equation*}
    Therefore, we have $h_{\mathcal{F}} > \frac{1}{\gamma}$.
    For $\mathcal{G}$, since $d_{n} + 1 < d_{n+1}$, we obtain
    \begin{equation*}
        P_{\mathcal{G}}\paren{\frac{1}{\gamma}} \geq \log\paren{T\sum_{n=1}^{\infty} \frac{1}{(d_{n}+1)^{2/\gamma}}} \geq \log\paren{T\sum_{n=1}^{\infty} \frac{1}{d_{n+1}^{2/\gamma}}} \geq \log\paren{\frac{T}{c_{2}^{2/\gamma}} \sum_{n=N+1}^{\infty} \frac{1}{n^{2}}}.
    \end{equation*}
    Therefore, we have $h_{\mathcal{G}} > \frac{1}{\gamma}$ similarly.
\end{proof}

\Cref{thm:measure_poly_gamma_general_T_large} and \Cref{thm:MauldinUrbanski1999-6.1}
suggest the following conjecture.

\begin{conjecture}
    Suppose $T \geq c_{2}^{2/\gamma}\cdot (N+1)$. Then,
    \begin{equation*}
        0 < H_{h_{\mathcal{F}}}(J_{\mathcal{F}}) < \infty, \ 0 < H_{h_{\mathcal{G}}}(J_{\mathcal{G}}) < \infty.
    \end{equation*}
\end{conjecture}

\subsection{Strict Inequality Between Hausdorff and Packing Dimensions}

In this subsection, we give an analogue of \Cref{thm:MauldinUrbanski1999-6.2},
which indicates that there are numerous examples of linear/nonlinear CIFSs with \textbf{(C3)}.

\begin{thm}[\Cref{main3}]\label{thm:dimension_gap_poly}
    Suppose $\gamma > 1$.
    For each $l \in \mathbb{N}$, define $\mathcal{F}^{(l)}$ and $\mathcal{G}^{(l)}$ by
    \begin{equation*}
        \mathcal{F}^{(l)} := \mathcal{F}(\set{d_{n}}_{n=l}^{\infty}, T), \quad \mathcal{G}^{(l)} := \mathcal{G}(\set{d_{n}}_{n=l}^{\infty}, T).
    \end{equation*}
    Then, there exists $q = q(\mathbf{d}, T) \in \mathbb{N}$ such that
    for all $l \in \mathbb{N}$ with $l \geq q$, we have
    \begin{align*}
        \dimh{J_{\mathcal{F}^{(l)}}} &< \dimlb{J_{\mathcal{F}^{(l)}}} \leq \dimub{J_{\mathcal{F}^{(l)}}} = \dimp{J_{\mathcal{F}^{(l)}}}, \\
        \dimh{J_{\mathcal{G}^{(l)}}} &< \dimlb{J_{\mathcal{G}^{(l)}}} \leq \dimub{J_{\mathcal{G}^{(l)}}} = \dimp{J_{\mathcal{G}^{(l)}}}.
    \end{align*}
\end{thm}

\begin{proof}
    By \Cref{thm:measure_poly_f},
    if $\gamma > 1$, we have
    \begin{equation*}
        \theta_{\mathcal{F}} = \frac{1}{2\gamma} < \frac{1}{\gamma + 1}
    \end{equation*}
    and $P_{\mathcal{F}}(1/(\gamma+1)) < \infty$.
    Therefore, by \Cref{lem:pressure_of_f},
    we obtain
    \begin{equation*}
        T\sum_{n=1}^{\infty} \frac{1}{(d_{n}^{2}-1)^{\frac{1}{\gamma+1}}} < \infty.
    \end{equation*}
    Thus, there exists $q = q(\mathbf{d}, T) \in \mathbb{N}$ such that
    for all $l \in \mathbb{N}$ with $l \geq q$, we have
    \begin{equation*}
        T\sum_{n=l}^{\infty} \frac{1}{(d_{n}^{2}-1)^{\frac{1}{\gamma+1}}} < 1.
    \end{equation*}
    This inequality and \Cref{lem:pressure_of_f} imply
    \begin{equation*}
        P_{\mathcal{F}^{(l)}}\paren{\frac{1}{\gamma+1}} = \log\paren{T\sum_{n=l}^{\infty} \frac{1}{(d_{n}^{2}-1)^{\frac{1}{\gamma+1}}}} < 0.
    \end{equation*}
    Therefore, by \Cref{thm:hausdorff-dimension-and-pressure} and \Cref{prop:hereditarily_regular_poly},
    we have
    \begin{equation}\label{eq:dimension_gap_poly_1}
        \dimh{J_{\mathcal{F}^{(l)}}}  < \frac{1}{\gamma+1}.
    \end{equation}
    Since $\mathcal{F}_{x}\setminus \mathcal{F}^{(l)}_{x}$ is a finite set for any $x \in [0, 1] \cap J_{\mathcal{F}} \cap J_{\mathcal{F}^{(l)}}$,
    the same argument as in the proof of \Cref{cor:lower-box-dim-poly} implies
    \begin{equation}\label{eq:finite_removed}
        \dimlb{J_{\mathcal{F}^{(l)}}} \geq \dimlb{\mathcal{F}^{(l)}_{x}} = \dimlb{\mathcal{F}_{x}} = \frac{1}{\gamma+1}.
    \end{equation}
    Thus, by \Cref{thm:box-dimension-and-packing-dimension-biLipschitz-case}, \eqref{eq:dimension_gap_poly_1} and \eqref{eq:finite_removed},
    we obtain
    \begin{equation*}
        \dimh{J_{\mathcal{F}^{(l)}}} < \dimlb{J_{\mathcal{F}^{(l)}}} \leq \dimub{J_{\mathcal{F}^{(l)}}} = \dimp{J_{\mathcal{F}^{(l)}}}.
    \end{equation*}
    The same argument works for $\mathcal{G}^{(l)}$.
\end{proof}

\section{Comparison of the Hausdorff dimensions of \texorpdfstring{$J_{\mathcal{F}}$ and $J_{\mathcal{G}}$}{J\_F and J\_G}}\label{sec:comparison-of-dimension-f-and-g}
\subsection{Motivation}
In this subsection, we compare the Hausdorff dimension of $J_{\mathcal{F}}$ and that of $J_{\mathcal{G}}$.
In \Cref{sec:f-and-g,sec:relationship_between_increasing_rate_and_measures},
we showed that $\mathcal{F}$ and $\mathcal{G}$ have similar properties.
Here, we refer to the following two results from \cite{MauldinUrbanski1996}.

\begin{prop}[\cite{MauldinUrbanski1996} Proposition 4.4]\label{prop:MauldinUrbanski1996-4.4}
    Let $\mathcal{S}$ be a CIFS on a compact subset $X$ of $\mathbb{R}^{d}$.
    Suppose that $\Leb_{d}\paren{\Int(X) \setminus J_{\mathcal{S}}^{(1)}} > 0$.
    Then, there exists $\mu \in (0, 1)$ such that for all $n \in \mathbb{N}$,
    $\Leb_{d}\paren{J_{\mathcal{S}}^{(n)}} \leq \mu^{n} \Leb_{d}\paren{X}$.
    In particular, $\Leb_{d}\paren{J_{\mathcal{S}}} = 0$.
\end{prop}

\begin{thm}[\cite{MauldinUrbanski1996} Theorem 4.5]\label{thm:MauldinUrbanski1996-4.5}
    Let $\mathcal{S}$ be a CIFS on a compact set $X$ of $\mathbb{R}^{d}$.
    \begin{enumerate}
        \item If $\mathcal{S}$ is regular and $\Leb_{d}\paren{\Int(X) \setminus J_{\mathcal{S}}^{(1)}} > 0$,
            then $h_{\mathcal{S}} < d$.
        \item If $\Leb_{d}\paren{X \setminus J_{\mathcal{S}}^{(1)}} = 0$,
            then $\mathcal{S}$ is regular and $0 < \Leb_{d}\paren{J_{\mathcal{S}}} = \Leb_{d}\paren{X} < \infty$.
        In particular, $h_{\mathcal{S}} = d$ and $m_{\mathcal{S}} = \Leb_{d}|_{X} / \Leb_{d}\paren{X}$.
    \end{enumerate}
\end{thm}

These results show that the property of $J_{\mathcal{S}}^{(1)}$ has a strong influence
on the measure-theoretic properties of the limit set $J_{\mathcal{S}}$.
Since $J_{\mathcal{F}}^{(1)} = J_{\mathcal{G}}^{(1)}$ by \Cref{prop:comparison_of_f_and_g},
it is natural to expect that $J_{\mathcal{F}}$ and $J_{\mathcal{G}}$ have similar properties.

On the other hand,
it is interesting to investigate how different the properties of $\mathcal{F}$ and $\mathcal{G}$ are.
Therefore, in this section, we focus on the Hausdorff dimensions of $J_{\mathcal{F}}$ and $J_{\mathcal{G}}$,
and make the differences between them clear.

Throughout this section, we assume that both $\mathcal{F} = \mathcal{F}(\mathbf{d}, T)$ and $\mathcal{G} = \mathcal{G}(\mathbf{d}, T)$
are CIFSs for $\mathbf{d} = \set{d_{n}}_{n=1}^{\infty} \in \mathcal{D}$ and $T \in \mathbb{N}$.

\subsection{The Case of Irregular Systems}
We first consider the case where either $\mathcal{F}$ or $\mathcal{G}$ is irregular.

\begin{thm}[\Cref{main4}]\label{thm:comparison-of-dimension-f-and-g-irregular}
    \begin{enumerate}
        \item Suppose $\mathcal{F}$ is irregular.
            Then, we have $h_{\mathcal{F}} \leq h_{\mathcal{G}}$.
        \item Suppose $\mathcal{G}$ is irregular.
            Then, we have $h_{\mathcal{G}} \leq h_{\mathcal{F}}$.
    \end{enumerate}
\end{thm}

\begin{proof}
    We prove (1). Since $\mathcal{F}$ is irregular,
    we have $P_{\mathcal{F}}(h_{\mathcal{F}}) < 0$.
    Therefore, \Cref{thm:comparison_of_f_and_g,thm:hausdorff-dimension-and-pressure} imply
    \begin{equation*}
        h_{\mathcal{F}}=\theta_{\mathcal{F}} = \theta_{\mathcal{G}} \leq h_{\mathcal{G}}.
    \end{equation*}
    The same argument works for (2).
\end{proof}

\subsection{The Case of \texorpdfstring{$T=1$}{T=1}}

We next consider the case where $T=1$.
When $T=1$, it is easy to estimate the pressure of $\mathcal{G}$ and compare $h_{\mathcal{F}}$ and $h_{\mathcal{G}}$.

\begin{lem}\label{lem:pressure_ineq_T1}
    Let $\mathbf{d} = \set{d_{n}}_{n=1}^{\infty}\in\mathcal{D}$ and $T=1$.
    For $t \geq 0$, we have
    \begin{equation*}
        \log\paren{\sum_{n=1}^{\infty} \frac{1}{(d_{n}+1)^{2t}}} \leq P_{\mathcal{G}}(t) \leq \log\paren{\sum_{n=1}^{\infty} \frac{1}{d_{n}^{2t}}} \leq \log\paren{\sum_{n=1}^{\infty} \frac{1}{(d_{n}^{2}-1)^{t}}} = P_{\mathcal{F}}(t).
    \end{equation*}
\end{lem}

\begin{proof}
    As for the second inequality, we have
    \begin{align}
        P_{\mathcal{G}}(t) &= \lim_{N\to\infty} \frac{1}{N}\log\paren{\sum_{(n_{1}, \dots, n_{N})\in\mathbb{N}^{N}} \left|\paren{g_{d_{n_{1}}} \circ g_{d_{n_{2}}} \circ \dots \circ g_{d_{n_{N}}}}'(0)\right|^{t}} \notag \\
        &= \lim_{N\to\infty} \frac{1}{N}\log\paren{\sum_{(n_{1}, \dots, n_{N})\in\mathbb{N}^{N}} \prod_{j=1}^{N} [d_{n_{j}}, d_{n_{j+1}}, \dots, d_{n_{N}}]^{2t}} \notag \\
        &\leq \lim_{N\to\infty} \frac{1}{N}\log\paren{\sum_{(n_{1}, \dots, n_{N})\in\mathbb{N}^{N}} \prod_{j=1}^{N} \frac{1}{d_{n_{j}}^{2t}}} \notag \\
        &= \lim_{N\to\infty} \frac{1}{N}\log\paren{\paren{\sum_{n=1}^{\infty} \frac{1}{d_{n}^{2t}}}^{N}} \notag \\
        &= \log\paren{\sum_{n=1}^{\infty} \frac{1}{d_{n}^{2t}}}. \notag 
    \end{align}
    The remaining inequalities follow from \Cref{lem:pressure_of_f,lem:pressure_ineq_g}.
\end{proof}

\begin{thm}[\Cref{main5}]\label{thm:main_theorem_dim1_T1}
    Let $\mathbf{d} \in \mathcal{D}$.
    If $\mathcal{G} = \mathcal{G}(\mathbf{d}, 1)$ is a regular CIFS,
    then we have $h_{\mathcal{F}\paren{\mathbf{d}, 1}} > h_{\mathcal{G}\paren{\mathbf{d}, 1}}$.
\end{thm}

\begin{proof}
    By the regularity of $\mathcal{G}$, we have $P_{\mathcal{G}}(h_{\mathcal{G}}) = 0$.
    Therefore, it follows from \Cref{lem:pressure_ineq_T1} and \Cref{thm:comparison_of_f_and_g} that
    $0 < P_{\mathcal{F}}(h_{\mathcal{G}}) < \infty$.
    Hence, we have $h_{\mathcal{F}} > h_{\mathcal{G}}$ by \Cref{thm:hausdorff-dimension-and-pressure}.
\end{proof}

\subsection{The Case of \texorpdfstring{$T \geq 2$}{T geq 2}}
In this subsection, we consider the case where $T \geq 2$.
We show that the inequalities for $P_{\mathcal{F}}$ and $P_{\mathcal{G}}$ differ from those in \Cref{lem:pressure_ineq_T1}.
We also give a sufficient condition for
$h_{\mathcal{F}} > h_{\mathcal{G}}$ when $T=4$.

We begin with the definition of a key function for establishing the inequality between the pressure functions.

\begin{definition}
    Let $a \geq 2, t > 0, 0 < r \leq 1$ and $T \in \mathbb{N}$.
    We define the function $\phi_{a, t, T}(z)$ on $\overline{\mathbb{D}}$ by
    \begin{equation*}
        \phi_{a, t, T}(z) := \sum_{j=0}^{T-1} \abs{ae\paren{j/T} + z}^{-2t} \quad (z \in \overline{\mathbb{D}}).
    \end{equation*}
    We also define the function $\psi_{r, a, t, T}(\theta)$ on $\mathbb{R}$ by
    \begin{equation*}
        \psi_{r, a, t, T}(\theta) := \phi_{a, t, T}(-re(\theta)) = \sum_{j=0}^{T-1} \abs{ae\paren{j/T} - re(\theta)}^{-2t} \quad (\theta \in \mathbb{R}).
    \end{equation*}
\end{definition}

\begin{lem}\label{lem:minimum_and_maximum_of_psi}
    Let $a \geq 2, t > 0, 0 < r \leq 1$ and $T \in \mathbb{N}$.
    Then, we have
    \begin{equation*}
        \argmax_{\theta \in \mathbb{R}} \psi_{r, a, t, T}(\theta) = \set{\frac{k}{T} \mid k \in \mathbb{Z}}.
    \end{equation*}
\end{lem}

\begin{proof}
    Fix $a \geq 2, t > 0, 0 < r \leq 1$ and $T \in \mathbb{N}$.
    Define $c_{k}$ by
    \begin{equation*}
        c_{k} := \begin{cases}
            1 & k = 0,\\
            \frac{\prod_{j=0}^{k-1}(t+j)}{k!} & k \in \mathbb{N}.
        \end{cases}
    \end{equation*}
    Note that $(1 - z)^{-t} = \sum_{k=0}^{\infty} c_{k} z^{k}$ for $z \in \mathbb{D}$.
    For $\theta \in \mathbb{R}$, we have
    \begin{align*}
        \psi_{r, a, t, T}(\theta) &= \sum_{j=0}^{T-1} \abs{ae(j/T) - re(\theta)}^{-2t} \\
        &= a^{-2t} \sum_{j=0}^{T-1}\abs{1 - \frac{re(\theta - j/T)}{a}}^{-2t} \\
        &= a^{-2t} \sum_{j=0}^{T-1}\set{\paren{1 - \frac{re\paren{\theta - j/T}}{a}}\paren{1 - \frac{re\paren{j/T-\theta}}{a}}}^{-t} \\
        &= a^{-2t} \sum_{j=0}^{T-1}\paren{\sum_{k=0}^{\infty}\frac{c_{k}r^{k}}{a^{k}}e\paren{k(\theta - j/T)}}\paren{\sum_{l=0}^{\infty}\frac{c_{l}r^{l}}{a^{l}}e\paren{l(j/T - \theta)}} \\
        &= a^{-2t} \sum_{j=0}^{T-1} \sum_{k=0}^{\infty}\sum_{l=0}^{\infty} \frac{c_{k}c_{l}r^{k+l}}{a^{k+l}} e\paren{(k-l)(\theta - j/T)} \\
        &= a^{-2t} \sum_{j=0}^{T-1}\sum_{p\in\mathbb{Z}}\sum_{k=0}^{\infty}\frac{c_{k}c_{k+\abs{p}}r^{2k+\abs{p}}}{a^{2k+\abs{p}}} e\paren{p(\theta - j/T)}.
    \end{align*}
    Therefore, we obtain
    \begin{equation}\label{eq:fourier_equation_of_psi_0}
        \psi_{r, a, t, T}(\theta) = a^{-2t} \sum_{j=0}^{T-1}\sum_{p\in\mathbb{Z}}\widetilde{c}_{r, p}e\paren{p(\theta - j/T)},
    \end{equation}
    where \begin{equation}\label{eq:fourier_coefficients}
        \widetilde{c}_{r, p} := \sum_{k=0}^{\infty}\frac{c_{k}c_{k+\abs{p}}r^{2k+\abs{p}}}{a^{2k+\abs{p}}} \quad (p \in \mathbb{Z}).
    \end{equation}
    For each $p\in\mathbb{Z}$, there exist $p_{1}, p_{2} \in \mathbb{Z}$ such that
    $p = Tp_{1} - p_{2}$ and $0 \leq p_{2} < T$.
    Thus, for each $p \in \mathbb{Z}$, we have
    \begin{align}\label{eq:sum-of-exponentials}
        \sum_{j=0}^{T-1} e\paren{p(\theta - j/T)} &= e(p\theta)\sum_{j=0}^{T-1} e\paren{(Tp_{1} - p_{2})(- j/T)} \notag \\
        &= e\paren{p\theta} \sum_{j=0}^{T-1} e(p_{2}j/T) \notag \\
        &= \begin{cases}
            Te(p\theta) & p_{2} = 0, \\
            0 & \text{otherwise}.
        \end{cases}
    \end{align}
    Applying \eqref{eq:sum-of-exponentials} to \eqref{eq:fourier_equation_of_psi_0}, we obtain
    \begin{equation}\label{eq:fourier_equation_of_psi_1}
        \psi_{r, a, t, T}(\theta) = \frac{T}{a^{2t}} \sum_{p\in\mathbb{Z}} \widetilde{c}_{r, pT} e(pT\theta) = \frac{T}{a^{2t}}\paren{\widetilde{c}_{r, 0} + 2\sum_{p=1}^{\infty} \widetilde{c}_{r, pT} \cos(2\pi p T \theta)}.
    \end{equation}
    Since $\widetilde{c}_{r, p}$ is positive by \eqref{eq:fourier_coefficients},
    $\psi_{r, a, t, T}(\theta)$
    attains its maximum when $\cos(2\pi p T \theta) = 1$ for all $p \in \mathbb{N}$.
    Thus, we conclude that
    \begin{equation*}
        \argmax_{\theta \in \mathbb{R}} \psi_{r, a, t, T}(\theta) = \set{\frac{k}{T} \mid k \in \mathbb{Z}}.
    \end{equation*}
\end{proof}

\begin{cor}\label{cor:minimum_and_maximum_of_phi}
    Let $a \geq 2, t > 0$ and $T \in \mathbb{N}$.
    Then, we have
    \begin{equation*}
        \max_{z \in \overline{\mathbb{D}}} \phi_{a, t, T}(z) = \sum_{j=0}^{T-1} \abs{ae(j/T) - 1}^{-2t} = \sum_{j=0}^{T-1} (a^{2} + 1 - 2a\cos(2\pi j/T))^{-t}.
    \end{equation*}
\end{cor}

\begin{proof}
    Let $0 < r \leq 1$.
    By \Cref{lem:minimum_and_maximum_of_psi} and \eqref{eq:fourier_equation_of_psi_1},
    if $\abs{z} = r$, then we have
    \begin{equation*}
        \phi_{a, t, T}(z) \leq \psi_{r, a, t, T}\paren{0} = \frac{T}{a^{2t}} \sum_{p\in\mathbb{Z}} \widetilde{c}_{r, pT}.
    \end{equation*}
    Note that $\psi_{r, a, t, T}\paren{0} = \phi_{a, t, T}(-r)$.
    Since $\widetilde{c}_{r, pT}$ is increasing with respect to $r$ by \eqref{eq:fourier_coefficients},
    for $z \in \overline{\mathbb{D}} \setminus \set{0}$, we obtain
    \begin{align*}
        \phi_{a, t, T}(z) &\leq \phi_{a, t, T}(-\abs{z}) \leq \phi_{a, t, T}(-1) \\
        &= \sum_{j=0}^{T-1} \abs{ae(j/T) - 1}^{-2t} = \sum_{j=0}^{T-1} (a^{2} + 1 - 2a\cos(2\pi j/T))^{-t}.
    \end{align*}
    $\phi_{a, t, T}(0) \leq \phi_{a, t, T}(-1)$ also follows
    from the continuity of $\phi_{a, t, T}$.
\end{proof}

\begin{cor}\label{lem:pressure_ineq_g_general_T}
    Let $\mathbf{d} = \set{d_{n}}_{n=1}^{\infty}\in\mathcal{D}$ and $T \in \mathbb{N}$.
    Suppose $\mathcal{G} = \mathcal{G}(\mathbf{d}, T)$ is a CIFS.
    For $t \geq 0$, we have
    \begin{equation*}
        P_{\mathcal{G}}(t) \leq \log\paren{\sum_{n=1}^{\infty}\sum_{j=0}^{T-1}\paren{d_{n}^{2}+1-2d_{n}\cos\paren{2\pi j/T}}^{-t}}.
    \end{equation*}
\end{cor}

\begin{proof}
    Let $I := \set{e\paren{j/T}d_{n} \mid j \in \set{0, \dots, T-1}, n \in \mathbb{N}}$.
    Note that
    \begin{equation}\label{eq:pressure_expression_g_at_0}
        P_{\mathcal{G}}(t) = \lim_{N\to\infty} \frac{1}{N}\log\paren{\sum_{\omega\in I^{N}} \abs{g_{\omega}'(0)}^{t}} \quad (t \geq 0).
    \end{equation}
    by \Cref{lem:other-expressions-of-pressure-function}.
    \Cref{cor:minimum_and_maximum_of_phi} implies that
    for each $N \in \mathbb{N}$ and $t \geq 0$, we obtain
    \begin{align*}
         & \sum_{\omega \in I^{N}} \abs{g_{\omega}'(0)}^{t} \\
        =& \sum_{\omega \in I^{N}} \prod_{j=1}^{N}\abs{\bracket{\omega_{j}, \omega_{j+1}, \dots, \omega_{N}}}^{2t} \\
        =& \sum_{\omega_{2}\dots\omega_{N} \in I^{N-1}}\left[\sum_{\omega_{1}\in I}\set{\abs{[\omega_{1}, \omega_{2}, \dots, \omega_{N}]}^{2t}\prod_{j=2}^{N}\abs{\bracket{\omega_{j}, \omega_{j+1}, \dots, \omega_{N}}}^{2t}}\right] \\
        =& \sum_{\omega_{2}\dots\omega_{N} \in I^{N-1}}\left[\paren{\prod_{j=2}^{N}\abs{\bracket{\omega_{j}, \omega_{j+1}, \dots, \omega_{N}}}^{2t}}\sum_{n=1}^{\infty}\sum_{j=0}^{T-1}\abs{[e\paren{j/T}d_{n}, \omega_{2}, \dots, \omega_{N}]}^{2t}\right] \\
        \leq& \sum_{\omega_{2}\dots\omega_{N} \in I^{N-1}}\left[\paren{\prod_{j=2}^{N}\abs{\bracket{\omega_{j}, \omega_{j+1}, \dots, \omega_{N}}}^{2t}}\sum_{n=1}^{\infty}\max_{z \in \overline{\mathbb{D}}}\phi_{d_{n}, t, T}(z)\right] \\
        \leq& \sum_{\omega_{2}\dots\omega_{N} \in I^{N-1}}\left[\paren{\prod_{j=2}^{N}\abs{\bracket{\omega_{j}, \omega_{j+1}, \dots, \omega_{N}}}^{2t}}\sum_{n=1}^{\infty}\phi_{d_{n}, t, T}(-1)\right] \\
        \leq& \dots \\
        \leq& \paren{\sum_{n=1}^{\infty}\phi_{d_{n}, t, T}(-1)}^{N} \\
        \leq& \paren{\sum_{n=1}^{\infty}\sum_{j=0}^{T-1}\paren{d_{n}^{2}+1-2d_{n}\cos\paren{2\pi j/T}}^{-t}}^{N}.
    \end{align*}
    This inequality and \eqref{eq:pressure_expression_g_at_0} imply the desired result.
\end{proof}

We now focus on the case where $T=4$.

\begin{lem}\label{lem:delta_t_existence}
    For $0 < t < 1$ and $v \geq 0$, define
    \begin{equation*}
        q(t, v) := 2\cosh\paren{tv} + 2\paren{\cosh\paren{v}}^{-t}.
    \end{equation*}
    Then, for each $0 < t < 1$, there exists a unique $v > 0$ such that $q(t, v) = 4$.
    Also, denoting this $v$ by $\delta_{t}$, we have the following.
    \begin{enumerate}
        \item If $0 < v < \delta_{t}$, then $q(t, v) < 4$.
        \item If $v > \delta_{t}$, then $q(t, v) > 4$.
    \end{enumerate}
\end{lem}

\begin{proof}
    Fix $0 < t < 1$ and let $q_{t}(v) := q(t, v)$ for $v \geq 0$.
    We have
    \begin{align}
        q_{t}'(v) &= 2t\paren{\sinh\paren{tv} - \paren{\cosh\paren{v}}^{-t-1}\sinh\paren{v}},\label{eq:derivative_of_q_t} \\
        q_{t}''(v) &= 2t\paren{t\cosh\paren{tv} + \paren{t+1}\paren{\cosh\paren{v}}^{-t-2}\paren{\sinh\paren{v}}^{2} - \paren{\cosh\paren{v}}^{-t}}. \label{eq:second_derivative_of_q_t}
    \end{align}
    In particular, $q_{t}'(0) = 0$ and $q_{t}''(0) = 2t(t-1) < 0$.
    Therefore, there exists $\varepsilon > 0$ such that
    \begin{equation}
        q_{t}'(v) < 0  \quad (0 < v < \varepsilon).\label{eq:q_t_derivative_negative_near_0}
    \end{equation}
    Let $v_{*} := \inf\set{v > 0 \mid q_{t}'(v) > 0}$.
    Note that $\set{v > 0 \mid q_{t}'(v) > 0} \ne \emptyset$
    since $q_{t}(v) \to \infty$ as $v \to \infty$,
    and that $v_{*} \geq \varepsilon$ by \eqref{eq:q_t_derivative_negative_near_0}.
    We have 
    \begin{equation}\label{eq:q_t_derivative_at_minimum_point}
        q_{t}'(v_{*}) = 0
    \end{equation}
    by the continuity of $q_{t}'$.
    We also have
    \begin{equation}\label{eq:second_derivative_at_minimum_point}
        q_{t}''\paren{v_{*}} \geq 0
    \end{equation}
    by the definition of $v_{*}$ and \eqref{eq:q_t_derivative_at_minimum_point}.

    Since $0 < \tanh(v) < 1$ for $v > 0$,
    \eqref{eq:derivative_of_q_t} and \eqref{eq:q_t_derivative_at_minimum_point} imply that
    for $v > 0$, we have
    \begin{align}\label{eq:third_derivative_of_q_t_and_first_derivative_relation}
         \ & q_{t}'''(v) \\
        = \ & 2t \Bigl\{ t^{2}\sinh(tv) - (t+1)(t+2)\paren{\cosh(v)}^{-t-3}\paren{\sinh(v)}^{3}  \notag \\
        & \quad \quad \quad  + 2(t+1)\paren{\cosh(v)}^{-t-1}\sinh(v) + t\paren{\cosh(v)}^{-t-1}\sinh(v) \Bigr\} \notag \\
        = \ & 2t \Bigl\{t^{2}\sinh(tv)   \notag \\
        & \quad \quad \quad  - (t+1)(t+2)\paren{\cosh(v)}^{-t}\paren{\tanh(v)}^{3} + (3t+2)\paren{\cosh(v)}^{-t}\tanh(v) \Bigr\} \notag \\
        >\ & 2t\Bigl\{t^{2}\sinh(tv)   \notag \\
        & \quad \quad \quad  - (t+1)(t+2)\paren{\cosh(v)}^{-t}\tanh(v) + (3t+2)\paren{\cosh(v)}^{-t}\tanh(v) \Bigr\} \notag \\
        = \ & 2t^{3}\set{\sinh(tv)  - \paren{\cosh(v)}^{-t}\tanh(v)} \notag \\
        = \ & t^{2} \cdot 2t\paren{\sinh\paren{tv} - \paren{\cosh\paren{v}}^{-t-1}\sinh\paren{v}} \notag \\
        = \ & t^{2} q_{t}'(v). \notag
    \end{align}
    \eqref{eq:q_t_derivative_at_minimum_point}, \eqref{eq:second_derivative_at_minimum_point} and
    \eqref{eq:third_derivative_of_q_t_and_first_derivative_relation} imply
    \begin{equation*}
        \begin{cases}
            q_{t}'(v) < 0 & (0 < v < v_{*}), \\
            q_{t}'(v) > 0 & (v > v_{*}).
        \end{cases}
    \end{equation*}
    Therefore, since $q_{t}(0) = 4$ and $q_{t}(v) \to \infty$ as $v \to \infty$,
    there exists a unique $\delta_{t} > 0$ such that $q_{t}(\delta_{t}) = 4$
    and the statements (1) and (2) hold.
\end{proof}

\begin{lem}\label{lem:delta_t_strictly_decreasing}
    $\delta_{t}$ in \Cref{lem:delta_t_existence} is strictly decreasing with respect to $t \in (0, 1)$.
\end{lem}

\begin{proof}
    For $0 < t < 1$ and $v > 0$, we have
    \begin{align}
        \frac{\partial q}{\partial t}(t, v) &= 2v\sinh(tv) - 2(\cosh(v))^{-t}\log(\cosh(v)) \notag \\
        &= 2\set{v\sinh(tv) +\frac{(\cosh(v))^{-t}\log\paren{\paren{\cosh(v)}^{-t}}}{t}}. \label{eq:partial_qt_partial_t}
    \end{align}
    Fix $s \in (0, 1)$. By the definition of $\delta_{s}$, we have
    \begin{equation}\label{eq:cosh_delta_s_power_minus_s}
        \paren{\cosh(\delta_{s})}^{-s} = 2 - \cosh(s\delta_{s}).
    \end{equation}
    Therefore, \eqref{eq:partial_qt_partial_t} and \eqref{eq:cosh_delta_s_power_minus_s} imply
    \begin{equation}\label{eq:partial_qt_partial_t_at_delta_s}
        \frac{\partial q}{\partial t}\paren{s, \delta_{s}} = \frac{2}{s}\set{s\delta_{s}\sinh(s\delta_{s}) + (2 - \cosh(s\delta_{s}))\log(2 - \cosh(s\delta_{s}))}.
    \end{equation}
    For $0 < u < u_{0}$, where $u_{0} > 0$ satisfies $\cosh(u_{0}) = 2$,
    we define the function $Q(u)$ by
    \begin{equation*}
        Q(u) := u\sinh(u) + \paren{2 - \cosh(u)}\log(2 - \cosh u) \quad (0 < u < u_{0}).
    \end{equation*}
    Since $0 < 2 - \cosh(u) < 1$ for $0 < u < u_{0}$, we have
    \begin{align*}
        Q'(u) &= \sinh(u) + u\cosh(u) - \sinh(u)\set{1 + \log(2 - \cosh(u))} \\
        &= u\cosh(u) - \sinh(u) \log(2 - \cosh(u)) > 0.
    \end{align*}
    Therefore, $Q(u) > \lim_{u \to 0} Q(u) = 0$ for $0 < u < u_{0}$.
    Hence, by $s > 0, \delta_{s} > 0$, \eqref{eq:cosh_delta_s_power_minus_s} and \eqref{eq:partial_qt_partial_t_at_delta_s},
    we obtain
    \begin{equation}\label{eq:partial_qt_partial_t_at_delta_s_positive}
        \frac{\partial q}{\partial t}\paren{s, \delta_{s}} = \frac{2}{s} Q(s\delta_{s}) > 0.
    \end{equation}

    By the proof of \Cref{lem:delta_t_existence},
    we have
    \begin{equation}\label{eq:partial_qt_partial_v_at_delta_s_positive}
        \frac{\partial q}{\partial v} (s, \delta_{s}) = q_{s}'(\delta_{s}) > 0.
    \end{equation}
    Therefore, by the implicit function theorem,
    there exists a neighborhood $I_{s} \subset (0, 1)$ of $s$ and
    a $C^{1}$-function $\widehat{\delta}$ defined on $I_{s}$ such that
    $q(t, \widehat{\delta}(t)) = 4$ for $t \in I_{s}$ and
    \begin{equation}\label{eq:derivative_of_delta_hat}
        \widehat{\delta}'(s) = - \frac{\frac{\partial q}{\partial t}\paren{s, \widehat{\delta}(s)}}{\frac{\partial q}{\partial v}\paren{s, \widehat{\delta}(s)}}.
    \end{equation}
    \Cref{lem:delta_t_existence} implies that $\widehat{\delta}(t) = \delta_{t}$ for $t \in I_{s}$.
    Thus, by \eqref{eq:partial_qt_partial_t_at_delta_s_positive}, \eqref{eq:partial_qt_partial_v_at_delta_s_positive} and \eqref{eq:derivative_of_delta_hat},
    we obtain
    \begin{equation}
        \frac{d}{dt} \delta_{t}\Big|_{t=s} = \widehat{\delta}'(s) = - \frac{\frac{\partial q}{\partial t}\paren{s, \widehat{\delta}(s)}}{\frac{\partial q}{\partial v}\paren{s, \widehat{\delta}(s)}} < 0.
    \end{equation}
    Since $s \in (0, 1)$ is arbitrary, we conclude that $\delta_{t}$ is strictly decreasing with respect to $t \in (0, 1)$.
\end{proof}

\begin{cor}\label{cor:delta_tilde_existence}
    For $0 < t < 1$ and $0 < x < 1$, define $\widetilde{q}_{t}(x)$ by
    \begin{equation*}
        \widetilde{q}_{t}(x) := \paren{\frac{1+x}{1-x}}^{t} + \paren{\frac{1-x}{1+x}}^{t} + 2\paren{\frac{1-x^2}{1+x^2}}^{t}.
    \end{equation*}
    Then, there exists $\widetilde{\delta}_{t} > 0$ such that
    if $0 < x < \widetilde{\delta}_{t}$, then $\widetilde{q}_{t}(x) < 4$.
    Moreover, $\widetilde{\delta}_{t}$ is strictly decreasing with respect to $t \in (0, 1)$.
\end{cor}

\begin{proof}
    Let $e^{v} = \frac{1+x}{1-x}$. We then have
    \begin{equation*}
        \widetilde{q}_{t}(x) = \paren{e^v}^{t} + \paren{e^{-v}}^{t} + 2\paren{\frac{2}{e^{v} + e^{-v}}}^{t} = 2\cosh(tv) + 2\paren{\cosh(v)}^{-t} = q(t, v).
    \end{equation*}
    Therefore, setting
    \begin{equation*}
        \widetilde{\delta}_{t} := \frac{e^{\delta_{t}} - 1}{e^{\delta_{t}} + 1},
    \end{equation*}
    $0 < v < \delta_{t}$ is equivalent to $0 < x < \widetilde{\delta}_{t}$.
    Thus, \Cref{lem:delta_t_existence} implies that
    if $0 < x < \widetilde{\delta}_{t}$, then $\widetilde{q}_{t}(x) < 4$.
    Moreover, since
    \begin{equation*}
        \widetilde{\delta}_{t} = 1 - \frac{2}{e^{\delta_{t}} + 1},
    \end{equation*}
    $\widetilde{\delta}_{t}$ is strictly increasing with respect to $\delta_{t}$.
    Therefore, by \Cref{lem:delta_t_strictly_decreasing}
    $\widetilde{\delta}_{t}$ is strictly decreasing with respect to $t \in (0, 1)$.
\end{proof}

We set $L(t)$ for $0 < t < 1$ by
\begin{equation}\label{eq:def_L_hG}
    L(t) := \max\set{2, \widetilde{\delta}_{t}^{-1}} \geq 2 \quad (0 < t < 1),
\end{equation}
where $\widetilde{\delta}_{t}$ is as in \Cref{cor:delta_tilde_existence}.
Note that $L(t)$ is non-decreasing with respect to $t \in (0, 1)$.

\begin{thm}\label{thm:value_of_L_at_half}
    If $0 < t \leq 1/2$, then $L(t) = 2$.
    Moreover, if $1/2 < t < 1$, then
    \begin{equation*}
        L(t) \leq 1 + \frac{1+t}{1-t + \sqrt{2(1-t)}}.
    \end{equation*}
\end{thm}

\begin{proof}
    For $0 < t < 1$ and $v \geq 0$, we have
    \begin{align}
        \cosh(tv) &= \sum_{n=0}^{\infty} \frac{(tv)^{2n}}{(2n)!} = 1 + \sum_{n=1}^{\infty} \frac{t^{2n} v^{2n}}{(2n)!} \notag \\
        &\leq t^{2}\cosh(v) + 1 - t^{2} \label{ineq:cosh}.
    \end{align}
    Also, for $0 < t < 1$ and $u \geq 0$, we have
    \begin{equation*}
        (1+u)^{-t} \leq 1 - t u + \frac{t(1+t)}{2} u^{2},
    \end{equation*}
    and then for $0 < t < 1, v \geq 0$, we have
    \begin{equation}\label{ineq:cosh-pow}
        (\cosh(v))^{-t} \leq 1 - t(\cosh(v) - 1) + \frac{t(1+t)}{2}(\cosh(v) - 1)^{2}.
    \end{equation}
    Therefore, by \eqref{ineq:cosh} and \eqref{ineq:cosh-pow},
    we obtain
    \begin{align*}
        q(t, v) &= 2\paren{\cosh(tv) + (\cosh(v))^{-t}} \\
        &\leq 2\paren{t^{2}\cosh(v) + 1 - t^{2} + 1 - t(\cosh(v) - 1) + \frac{t(1+t)}{2}(\cosh(v) - 1)^{2}} \\
        &= 4 + t(1+t)(\cosh(v) - 1)\paren{\cosh(v) - \frac{3 - t}{1 + t}}
    \end{align*}
    for $0 < t < 1$ and $v \geq 0$.
    Let $\xi_{t} >0$ be such that $\cosh(\xi_{t}) = \frac{3 - t}{1 + t}$.
    Note that
    \begin{equation*}
        e^{\xi_{t}} = \frac{3 - t}{1 + t} + \sqrt{\paren{\frac{3 - t}{1 + t}}^{2} - 1} = \frac{3 - t + 2\sqrt{2(1 - t)}}{1 + t}.
    \end{equation*}
    Since $q(t, \xi_{t}) \leq 4$,
    we have $\delta_{t} \geq \xi_{t}$ by \Cref{lem:delta_t_existence}.
    Thus,
    \begin{equation*}
        \paren{\widetilde{\delta}_{t}}^{-1} = \frac{e^{\delta_{t}}+1}{e^{\delta_{t}} - 1} = 1 + \frac{2}{e^{\delta_{t}} - 1} \leq 1 + \frac{2}{e^{\xi_{t}} - 1} = 1 + \frac{1 + t}{1 - t + \sqrt{2(1 - t)}}
    \end{equation*}
    holds for $0 < t < 1$.
    In particular,
    $\widetilde{L}(t) := 1 + \frac{1 + t}{1 - t + \sqrt{2(1 - t)}}$ is
    strictly increasing with respect to $t \in (0, 1)$
    and $\widetilde{L}(1/2) = 2$.
    Therefore, since
    \begin{equation*}
        2 \leq L(t) = \max\set{2, \paren{\widetilde{\delta}_{t}}^{-1}} \leq\max\set{{2, \widetilde{L}(t)}},
    \end{equation*}
    we conclude that $L(t) = 2$ for $0 < t \leq 1/2$
    and that $L(t) \leq \widetilde{L}(t)$ for $1/2 < t < 1$.
\end{proof}

\begin{thm}[First half of \Cref{main6}]\label{thm:comparison_of_dimensions_T_4}
    Let $\mathbf{d} \in \mathcal{D}$ and suppose that $\mathcal{F} = \mathcal{F}\paren{\mathbf{d}, 4}$ is a CIFS.
    Assume that $\mathcal{G} = \mathcal{G}\paren{\mathbf{d}, 4}$ is a regular CIFS
    and $0 < h_{\mathcal{G}} < 1$.
    Then,
    if $d_{1} \geq L\paren{h_{\mathcal{G}}}$, we have $h_{\mathcal{F}} > h_{\mathcal{G}}$.
\end{thm}

\begin{proof}
    We use the same notation as in the proof of \Cref{cor:delta_tilde_existence}
    for $\widetilde{q}_{t}, \widetilde{\delta}_{t}$.
    By \Cref{lem:pressure_ineq_g_general_T},
    for $t > 0$, we obtain
    \begin{align*}
        P_{\mathcal{G}}(t) &\leq \log\paren{\sum_{n=1}^{\infty}\set{\paren{d_{n}-1}^{-2t} + \paren{d_{n}+1}^{-2t} + 2\paren{d_{n}^{2}+1}^{-t}}} \\
        &= \log\paren{\sum_{n=1}^{\infty} \frac{1}{\paren{d_{n}^{2}-1}^{t}}\set{\paren{\frac{d_{n}-1}{d_{n}+1}}^{t} + \paren{\frac{d_{n}+1}{d_{n}-1}}^{t} + 2\paren{\frac{d_{n}^{2}-1}{d_{n}^{2}+1}}^{t}}} \\
        &= \log\paren{\sum_{n=1}^{\infty} \frac{1}{\paren{d_{n}^{2}-1}^{t}}\widetilde{q}_{t}\paren{\frac{1}{d_{n}}}}.
    \end{align*}
    If $d_{1} \geq L\paren{h_{\mathcal{G}}}$, then for any $n \in \mathbb{N}$ with $n \geq 2$,
    we have
    \begin{equation*}
        0 < \frac{1}{d_{n}} < \frac{1}{d_{1}} \leq \frac{1}{L} \leq \widetilde{\delta}_{h_{\mathcal{G}}}.
    \end{equation*}
    Therefore, if $n \geq 2$, then
    $0 < \widetilde{q}_{h_{\mathcal{G}}}\paren{\frac{1}{d_{n}}} < 4$.
    Also, since $0 < 1/d_{1} \leq 1/L \leq \widetilde{\delta}_{h_{\mathcal{G}}}$,
    we have $0 < \widetilde{q}_{h_{\mathcal{G}}}\paren{\frac{1}{d_{1}}} \leq 4$.
    From the above, \Cref{lem:pressure_of_f} and \Cref{thm:comparison_of_f_and_g}(1), we obtain that
    \begin{equation*}
        0 = P_{\mathcal{G}}(h_{\mathcal{G}}) \leq \log\paren{\sum_{n=1}^{\infty} \frac{\widetilde{q}_{h_{\mathcal{G}}}\paren{\frac{1}{d_{n}}}}{\paren{d_{n}^{2}-1}^{h_{\mathcal{G}}}}} < \log\paren{\sum_{n=1}^{\infty} \frac{4}{\paren{d_{n}^{2}-1}^{h_{\mathcal{G}}}}} = P_{\mathcal{F}}(h_{\mathcal{G}}) < \infty
    \end{equation*}
    holds. Therefore, $h_{\mathcal{F}} > h_{\mathcal{G}}$.
\end{proof}

\begin{cor}[Second half of \Cref{main6}]\label{cor:comparison_of_dimensions_T_4_at_half}
    Let $\mathbf{d} \in \mathcal{D}$ and suppose that $\mathcal{F} = \mathcal{F}\paren{\mathbf{d}, 4}$ is a CIFS.
    Also, suppose that $\mathcal{G} = \mathcal{G}\paren{\mathbf{d}, 4}$ is a regular CIFS
    and satisfies $0 < h_{\mathcal{G}} \leq 1/2$.
    Then,
    $h_{\mathcal{F}} > h_{\mathcal{G}}$ holds.
\end{cor}

\begin{proof}
    By \Cref{thm:value_of_L_at_half},
    $0 < h_{\mathcal{G}} \leq 1/2$ implies $L(h_{\mathcal{G}}) = 2$.
    Therefore, since $\mathbf{d} \in \mathcal{D}$, we have $d_{1} \geq 2 = L(h_{\mathcal{G}})$.
    Hence, by \Cref{thm:comparison_of_dimensions_T_4},
    $h_{\mathcal{F}} > h_{\mathcal{G}}$ holds.
\end{proof}

\begin{problem}
    Let $\mathbf{d} \in \mathcal{D}$ and $T \in \mathbb{N}$.
    Suppose that $\mathcal{F} = \mathcal{F}(\mathbf{d}, T)$ is a CIFS
    and that $\mathcal{G} = \mathcal{G}(\mathbf{d}, T)$ is a regular CIFS.
    Is it always true that $h_{\mathcal{F}} > h_{\mathcal{G}}$?
    If not, can one find explicit parameters $\mathbf{d}$ and $T$
    such that $h_{\mathcal{F}} \leq h_{\mathcal{G}}$?
\end{problem}

\section{Examples}\label{sec:examples}
We conclude the paper with some examples illustrating our results.

\subsection{Examples of Polynomial Growth Cases}
\begin{ex}[\Cref{main7}: Example where \textbf{(C1)} holds]\label{ex:C1_satisfied_example}
    For given $\gamma > 1$ and $T \in \mathbb{N}$,
    we set $\mathbf{d}^{(L)}$ for each $L \in \mathbb{N}$ by
    $\mathbf{d}^{(L)} := \set{(n+L)^{\gamma} + 1}_{n=1}^{\infty}$.
    Then, for any sufficiently large $L \in \mathbb{N}$,
    $\mathcal{F} = \mathcal{F}(\mathbf{d}^{(L)}, T)$ and $\mathcal{G} = \mathcal{G}(\mathbf{d}^{(L)}, T)$
    are both CIFSs and
    \begin{equation*}
        H_{h_{\mathcal{F}\paren{\mathbf{d}^{(L)}, T}}}\paren{J_{\mathcal{F}\paren{\mathbf{d}^{(L)}, T}}} = H_{h_{\mathcal{G}\paren{\mathbf{d}^{(L)}, T}}}\paren{J_{\mathcal{G}\paren{\mathbf{d}^{(L)}, T}}} = 0
    \end{equation*}
    hold.
\end{ex}

\begin{proof}
    It is easy to check $\mathbf{d}^{(L)} \in \mathcal{D}$ and $T \leq \pi/\arcsin\paren{1/d^{(L)}_{1}}$
    for sufficiently large $L \in \mathbb{N}$.
    Take such $L$.
    Then, \Cref{lem:cifs-sufficient-cond} implies that
    $\mathcal{F}(\mathbf{d}^{(L)}, T)$ and $\mathcal{G}(\mathbf{d}^{(L)}, T)$
    are both CIFSs.
    We have
    \begin{equation*}
        P_{\mathcal{F}\paren{\mathbf{d}^{(L)}, T}}(1/\gamma) \leq \log\paren{T\sum_{n=1}^{\infty} \paren{\frac{1}{\paren{d^{(L)}_{n}-1}^{2}}}^{1/\gamma}} = \log\paren{T\sum_{n=1}^{\infty} \frac{1}{(n+L)^{2}}}
    \end{equation*}
    by \Cref{lem:pressure_of_f}.
    Therefore, replacing $L$ by a larger integer if necessary,
    we obtain $P_{\mathcal{F}\paren{\mathbf{d}^{(L)}, T}}(1/\gamma) < 0$.
    Hence,
    we have $h_{\mathcal{F}\paren{\mathbf{d}^{(L)}, T}} < 1/\gamma$ by the regularity of $\mathcal{F}\paren{\mathbf{d}^{(L)}, T}$,
    and then \Cref{thm:measure_poly_f} indicates
    \begin{equation*}
        H_{h_{\mathcal{F}\paren{\mathbf{d}^{(L)}, T}}}\paren{J_{\mathcal{F}\paren{\mathbf{d}^{(L)}, T}}} = 0.
    \end{equation*}
    The same argument works for $\mathcal{G}\paren{\mathbf{d}^{(L)}, T}$.
\end{proof}

\begin{ex}[Example where \textbf{(C2)} holds]\label{ex:C2_satisfied_example_gamma1}
    Suppose $\set{d_{n}}_{n=1}^{\infty}\in\mathcal{D}$ satisfies \eqref{eq:polynomial_growth_conditions} for $\gamma = 1$.
    If $d_{1}^{2}-1 \leq T \leq \pi/\arcsin(1/d_{1})$,
    then $\dimp{J_{\mathcal{F}}} = h_{\mathcal{F}}$
    and $\Pi_{h_{\mathcal{F}}}(J_{\mathcal{F}}) = \infty$.
    In particular, if $d_{1} = 2$ and $T \in \set{3, 4, 5, 6}$,
    then $\dimp{J_{\mathcal{F}}} = h_{\mathcal{F}}$ and
    $\Pi_{h_{\mathcal{F}}}(J_{\mathcal{F}}) = \infty$ hold
    since $d_{1}^{2} - 1 = 3 \leq T \leq 6 = \pi/\arcsin\paren{1/d_{1}}$.
\end{ex}

\begin{proof}
    If $d_{1}^{2}-1 \leq T$ holds, \Cref{lem:pressure_of_f} and \Cref{prop:theta_poly} imply
    \begin{equation*}
        0 \leq \log\paren{\frac{T}{d_{1}^{2}-1}} < \log\paren{T\sum_{n=1}^{\infty} \frac{1}{(d_{n}^{2}-1)}} = P_{\mathcal{F}}(1) < \infty.
    \end{equation*}
    Thus, we have $h_{\mathcal{F}} > 1 = \gamma$ and then \Cref{cor:hausdorff-dim-equals-packing-dim-and-packing-measure-infty} implies
    $\dimp{J_{\mathcal{F}}} = h_{\mathcal{F}}$ and $\Pi_{h_{\mathcal{F}}}(J_{\mathcal{F}}) = \infty$.
\end{proof}

\begin{ex}[Example where $h_{\mathcal{F}} = 1/\gamma$ holds]
    Let $\mathbf{d} = \set{2n}_{n=1}^{\infty}$ and $T=2$.
    We then have
    \begin{equation*}
        h_{\mathcal{F}} = \dimp{J_{\mathcal{F}}} = 1, \quad 0 < H_{h_{\mathcal{F}}}\paren{J_{\mathcal{F}}} < \infty, \quad 0 < \Pi_{h_{\mathcal{F}}}\paren{J_{\mathcal{F}}} < \infty.
    \end{equation*}
\end{ex}

\begin{proof}
    Since $J_{\mathcal{F}}$ is contained in $[-1, 1]$,
    we have $\overline{J_{\mathcal{F}}} \subset [-1, 1]$.
    We show the opposite inclusion.
    Let $I := \set{2n \mid n \in \mathbb{N}} \cup \set{-2n \mid n \in \mathbb{N}}$.
    Take an arbitrary $x_{0} \in [-1, 1] \setminus \set{0}$.
    By \Cref{prop:image}, there exist $\omega_{1} \in I$ such that
    $x_{0} \in f_{\omega_{1}}\paren{\overline{\mathbb{D}}}$.
    Let $x_{1} := f_{\omega_{1}}^{-1}(x_{0}) \in [-1, 1]$.
    Then, if $x_{1} \ne 0$, by \Cref{prop:image} again,
    there exists $\omega_{2} \in I$ such that $x_{1} \in f_{\omega_{2}}\paren{\overline{\mathbb{D}}}$.
    Let $x_{2} := f_{\omega_{2}}^{-1}(x_{1}) \in [-1, 1]$.
    By repeating this procedure, if $x_{n} \ne 0$ for every $n \in \mathbb{N}$,
    we obtain a sequence $\omega = (\omega_{1}, \omega_{2}, \dots) \in I^{\mathbb{N}}$
    and a sequence $\set{x_{n}}_{n=0}^{\infty} \in [-1, 1]^{\mathbb{N}}$ such that
    $x_{n} = f_{\omega_{n}}^{-1}(x_{n-1})$ for each $n \in \mathbb{N}$.
    Therefore, we have $x_{0} = \paren{f_{\omega_{1}} \circ f_{\omega_{2}} \circ \dots \circ f_{\omega_{n}}}(x_{n})$
    for each $n \in \mathbb{N}$,
    and then $x_{0} \in J_{\mathcal{F}} = \bigcup_{\tau \in I^{\mathbb{N}}} \bigcap_{n=1}^{\infty} \paren{f_{\tau_{1}} \circ \dots \circ f_{\tau_{n}}}\paren{\overline{\mathbb{D}}}$.
    If $x_{n} = 0$ for some $n \in \mathbb{N}$,
    then $x_{0} = \paren{f_{\omega_{1}} \circ f_{\omega_{2}} \circ \dots \circ f_{\omega_{n}}}(0)$.
    Since we have
    \begin{equation*}
        \overline{J_{\mathcal{F}}} = J_{\mathcal{F}} \cup A
    \end{equation*}
    by \Cref{prop:closure-of-limit-set,prop:comparison_of_f_and_g}(2),
    where
    \begin{equation*}
        A := \set{0} \cup \bigcup_{N\in\mathbb{N}}\bigcup_{\omega \in I^{N}}f_{\omega}\paren{\set{0}},
    \end{equation*}
    we obtain $x_{0} \in \overline{J_{\mathcal{F}}}$.
    Thus, we obtain $\overline{J_{\mathcal{F}}} = [-1, 1]$ and $J_{\mathcal{F}} \cup A = [-1, 1]$.
    Since $A$ is countable and $H_{1}\paren{[-1, 1]}, \Pi_{1}\paren{[-1, 1]}\in (0, \infty)$,
    we conclude that
    \begin{equation*}
        h_{\mathcal{F}} = \dimp{J_{\mathcal{F}}} = 1, \quad 0 < H_{h_{\mathcal{F}}}\paren{J_{\mathcal{F}}} < \infty, \quad 0 < \Pi_{h_{\mathcal{F}}}\paren{J_{\mathcal{F}}} < \infty.
    \end{equation*}
\end{proof}

\begin{ex}[Example where \textbf{(C1)} or \textbf{(C2)} holds depending on $T$]
    For $\delta > 0$ with
    \begin{equation}\label{eq:delta_condition_for_g_packing_measure_infty}
        \sum_{n=1}^{\infty} \frac{1}{(2n+1+\delta)^{2}} > 0.2,
    \end{equation}
    let $d_{n} := 2n + 1 + \delta$ for $n \in \mathbb{N}$.
    Note that $\sum_{n=1}^{\infty} \frac{1}{(2n+1)^{2}} = \pi^{2}/8 - 1 \approx 0.2337$.
    Then, we have the following statements.
    \begin{enumerate}
        \item If $T=1, 2$, then $H_{h_{\mathcal{G}}}\paren{J_{\mathcal{G}}} = 0$.
        \item If $T = 5, 6$, then $\dimp{J_{\mathcal{G}}} = h_{\mathcal{G}}$ and $\Pi_{h_{\mathcal{G}}}\paren{J_{\mathcal{G}}} = \infty$.
    \end{enumerate}
\end{ex}

\begin{proof}
    We first prove (1).
    Let $T \in \set{1, 2}$. We set $\mathcal{G}|_{[-1, 1]} := \set{g|_{[-1, 1]} \mid g \in \mathcal{G}}$.
    We can check that $\mathcal{G}|_{[-1, 1]}$ is a CIFS on $[-1, 1]$ by the same argument as in the proof of \Cref{lem:cifs-sufficient-cond}.
    We also have $P_{\mathcal{G}|_{[-1, 1]}} = P_{\mathcal{G}}$ by \Cref{lem:other-expressions-of-pressure-function}.
    Therefore, $\mathcal{G}|_{[-1, 1]}$ is regular since $\mathcal{G}$ is regular by \Cref{prop:hereditarily_regular_poly}.
    We have
    \begin{equation*}
        \Int\paren{[-1, 1]} \setminus J_{\mathcal{G}|_{[-1, 1]}}^{(1)} \supset (-1, 1) \setminus \bracket{-\frac{1}{d_{1}-1}, \frac{1}{d_{1}-1}} = \paren{-1, -\frac{1}{1+\delta}} \cup \paren{\frac{1}{1+\delta}, 1}.
    \end{equation*}
    Hence, \Cref{thm:MauldinUrbanski1996-4.5} (1) implies $h_{\mathcal{G}|_{[-1, 1]}} < 1$.
    Therefore, we obtain $h_{\mathcal{G}} = h_{\mathcal{G}|_{[-1, 1]}} < 1$
    by $J_{\mathcal{G}} = J_{\mathcal{G}|_{[-1, 1]}}$,
    and then \Cref{thm:measure_poly_g} indicates
    \begin{equation*}
        H_{h_{\mathcal{G}}}\paren{J_{\mathcal{G}}} = 0.
    \end{equation*}

    We next prove (2).
    If $T = 5, 6$, then by \eqref{eq:delta_condition_for_g_packing_measure_infty}, we have
    \begin{equation*}
        0 \leq \log\paren{0.2 T} < \log\paren{T\sum_{n=1}^{\infty} \frac{1}{(2n+\delta+1)^{2}}} = P_{\mathcal{G}}(1) < \infty.
    \end{equation*}
    Thus, we have $h_{\mathcal{G}} > 1$ and then \Cref{cor:hausdorff-dim-equals-packing-dim-and-packing-measure-infty} indicates
    $\dimp{J_{\mathcal{G}}} = h_{\mathcal{G}}$ and $\Pi_{h_{\mathcal{G}}}\paren{J_{\mathcal{G}}} = \infty$.
\end{proof}

\begin{ex}[\Cref{main8}: Example where \textbf{(C2)} and $h_{\mathcal{F}} > h_{\mathcal{G}}$ hold]\label{ex:c2-and-comparison_of_dimensions_T4_example}
    Define $\mathbf{d} = \set{d_{n}}_{n=1}^{\infty}$ by
    \begin{equation*}
        d_{n} = \begin{cases}
            17 & (n=1), \\
            19 & (n=2), \\
            n^{3} & (n \geq 3).
        \end{cases}
    \end{equation*}
    Let $\mathcal{F} = \mathcal{F}(\mathbf{d}, 4), \ \mathcal{G} = \mathcal{G}(\mathbf{d}, 4)$.
    Then, we have the following statements.
    \begin{enumerate}
        \item $\dimp{J_{\mathcal{F}}} = h_{\mathcal{F}}$ and $\Pi_{h_{\mathcal{F}}}(J_{\mathcal{F}}) = \infty$.
        \item $\dimp{J_{\mathcal{G}}} = h_{\mathcal{G}}$ and $\Pi_{h_{\mathcal{G}}}(J_{\mathcal{G}}) = \infty$.
        \item $h_{\mathcal{G}} < h_{\mathcal{F}} < 1/2$.
    \end{enumerate}
\end{ex}

\begin{proof}
    Note that this $\mathbf{d}$ satisfies \eqref{eq:polynomial_growth_conditions}
    for $c_{1} = c_{2} = 1, \gamma =3, N=3$.
    We thus have (1) and (2) by \Cref{thm:measure_poly_gamma_general_T_large}.
    We show (3). Since we have
    \begin{equation*}
        \sum_{n=3}^{\infty}\frac{1}{d_{n}-1} = \sum_{n=3}^{\infty}\frac{n^{3}}{n^{3}(n^{3}-1)} \leq \frac{27}{26}\int_{2}^{\infty}x^{-3}dx = \frac{27}{208},
    \end{equation*}
    it follows that
    \begin{equation*}
        4\sum_{n=1}^{\infty}\frac{1}{d_{n}-1} \leq 4\paren{\frac{1}{16} + \frac{1}{18} + \frac{27}{208}} = \frac{116}{117} < 1.
    \end{equation*}
    Therefore, \Cref{prop:hereditarily_regular_poly} and \Cref{lem:pressure_of_f,lem:pressure_ineq_g} imply
    $h_{\mathcal{F}} < 1/2$ and $h_{\mathcal{G}} < 1/2$.
    Therefore, by \Cref{cor:comparison_of_dimensions_T_4_at_half},
    we obtain $h_{\mathcal{G}} < h_{\mathcal{F}} < 1/2$.
\end{proof}

\subsection{Construction of Irregular CIFSs}
In this subsection, we construct a family of irregular CIFSs.
This example is an analogue of Example 6.5 in \cite{MauldinUrbanski1999}.

For $\lambda > 1$, define the sequence $\mathbf{d}^{(\lambda)} = \set{d_{n}^{(\lambda)}}_{n\in\mathbb{N}}$ by
\begin{equation*}
    d^{(\lambda)}_{n} := \frac{2(n+2)\left(\log(n+2)\right)^{\lambda}}{\left(\log{3}\right)^{\lambda}}.
\end{equation*}
We have $d^{(\lambda)}_{1} = 6 > 2$ and
\begin{align*}
    d^{(\lambda)}_{n+1} - d^{(\lambda)}_{n} &= \frac{2}{(\log 3)^{\lambda}}\left((n+3)(\log(n+3))^{\lambda} - (n+2)(\log(n+2))^{\lambda}\right) \\
    &\geq \frac{2}{(\log 3)^{\lambda}}\left((n+3)(\log(n+2))^{\lambda} - (n+2)(\log(n+2))^{\lambda}\right) \\
    &= \frac{2}{(\log 3)^{\lambda}}(\log(n+2))^{\lambda} \\
    &\geq 2.
\end{align*}
Thus, we have $\mathbf{d}^{(\lambda)} \in \mathcal{D}$.
We also have
\begin{equation*}
    \limsup_{N\to\infty} \frac{\log{N}}{2\log{d_{N}^{(\lambda)}}} = \limsup_{N\to\infty} \frac{1}{2}\cdot\frac{\log{N}}{\log{(2(N+2))} + \lambda\log(\log(N+2)) - \lambda\log(\log 3)} = \frac{1}{2}.
\end{equation*}
Therefore, if $T \in \mathbb{N}$ satisfies $T \leq \pi/\arcsin(1/d^{(\lambda)}_{1})$,
then we have
$\theta := \theta_{\mathcal{F}\paren{\mathbf{d}^{(\lambda)}, T}} = \theta_{\mathcal{G}\paren{\mathbf{d}^{(\lambda)}, T}} = 1/2$
by \Cref{thm:theta_via_Dirichlet}.
We set
\begin{align*}
    s(\lambda) &:= \sum_{n=1}^{\infty}\frac{1}{\left(d^{(\lambda)}_{n}\right)^{2\theta}} = \sum_{n=1}^{\infty}\frac{1}{d^{(\lambda)}_{n}} \quad (\lambda > 1), \\
    \beta(x) &:= \frac{\left(\log 3\right)^{\lambda}}{2}\frac{1}{(x+1)(\log(x+1))^{\lambda}} \quad (x > 0).
\end{align*}
Since $1/d^{(\lambda)}_{n} = \beta(n+1)$ for $n \in \mathbb{N}$,
we have
\begin{equation*}
    \int_{2}^{\infty} \beta(x) dx \leq s(\lambda) \leq \int_{1}^{\infty} \beta(x) dx.
\end{equation*}
By simple calculations, we obtain
\begin{equation*}
    \int_{1}^{\infty} \beta(x) dx = \frac{\log 2}{2(\lambda - 1)}\paren{\frac{\log{3}}{\log{2}}}^{\lambda}, \quad \int_{2}^{\infty} \beta(x) dx = \frac{\log 3}{2(\lambda - 1)}.
\end{equation*}
Thus, we have
\begin{equation*}
    \frac{\log 3}{2(\lambda - 1)} \leq s(\lambda) \leq \frac{\log 2}{2(\lambda - 1)}\left(\frac{\log 3}{\log 2}\right)^{\lambda}.
\end{equation*}
Since $0.69 < \log{2} < 0.7$ and $1.09 < \log{3} < 1.1$, we obtain
\begin{equation*}
    s\paren{\frac{3}{2}} > 1, \quad s(2) < 1.
\end{equation*}

Hence, by the continuity of $s(\lambda)$ with respect to $\lambda > 1$,
there exists $\lambda_{0} \in (3/2, 2)$ such that $s(\lambda_{0}) = 1$.
We then have the following \Cref{prop:irregular_example_T1,prop:irregular_example_general_T}.

\begin{prop}\label{prop:irregular_example_T1}
    There exists $\lambda > 1$ such that
    $\mathcal{F}(\mathbf{d}^{(\lambda)}, 1)$ is regular,
    but
    $\mathcal{G}(\mathbf{d}^{(\lambda)}, 1)$ is not regular.
\end{prop}

\begin{proof}
    \Cref{lem:pressure_ineq_T1}, \Cref{thm:comparison_of_f_and_g} and $\theta = 1/2$ imply
    \begin{equation}\label{eq:pressure_inequality_at_lambda0}
        P_{\mathcal{G}\paren{\mathbf{d}^{(\lambda_{0})}, 1}}(\theta) \leq \log\paren{\sum_{n=1}^{\infty} \frac{1}{d^{(\lambda_{0})}_{n}}} =  0 < \log\paren{\sum_{n=1}^{\infty} \frac{1}{\left(\left(d^{(\lambda_{0})}_{n}\right)^{2}-1\right)^{1/2}}} = P_{\mathcal{F}\paren{\mathbf{d}^{(\lambda_{0})}, 1}}(\theta) < \infty.
    \end{equation}
    Also,
    \begin{equation*}
        \log\paren{\sum_{n=1}^{\infty} \frac{1}{\left(\left(d^{(\lambda)}_{n}\right)^{2}-1\right)^{1/2}}}, \quad \log\paren{\sum_{n=1}^{\infty} \frac{1}{d^{(\lambda)}_{n}}}
    \end{equation*}
    are continuous and strictly decreasing with respect to $\lambda > 1$.
    Therefore, if we take $\delta > 0$ sufficiently small, then
    \eqref{eq:pressure_inequality_at_lambda0} indicates
    \begin{equation*}
        P_{\mathcal{G}\paren{\mathbf{d}^{(\lambda_{0} + \delta)}, 1}}(\theta) < 0 < P_{\mathcal{F}\paren{\mathbf{d}^{(\lambda_{0} + \delta)}, 1}}(\theta) < \infty.
    \end{equation*}
    Thus, for $\lambda = \lambda_{0} + \delta$,
    $\mathcal{F}\paren{\mathbf{d}^{(\lambda)}, 1}$ is regular.
    However, $\mathcal{G}\paren{\mathbf{d}^{(\lambda)}, 1}$ is not regular.
\end{proof}

\begin{prop}\label{prop:irregular_example_general_T}
    There exist $\mathbf{d}$ and $T \in \mathbb{N}$ such that
    $\mathcal{F}\paren{\mathbf{d}, T}, \mathcal{G}\paren{\mathbf{d}, T}$ are both irregular.
\end{prop}

\begin{proof}
    Take $T \in \mathbb{N}$ so that $T \leq \pi/\arcsin\paren{1/d^{(\lambda_{0})}_{1}}$.
    Since $s(\lambda_{0}) = 1$, using an argument similar to that in \Cref{thm:comparison_of_f_and_g}(1),
    we obtain
    $P_{\mathcal{F}\paren{\mathbf{d}^{(\lambda_{0})}, T}}(\theta) < \infty$
    and
    $P_{\mathcal{G}\paren{\mathbf{d}^{(\lambda_{0})}, T}}(\theta) < \infty$.
    These inequalities imply that $\mathcal{F}\paren{\mathbf{d}^{(\lambda_{0})}, T}$ and $\mathcal{G}\paren{\mathbf{d}^{(\lambda_{0})}, T}$ are not hereditarily regular.
    Therefore, by taking an appropriate subsequence $\mathbf{d}'$ of $\mathbf{d}^{(\lambda_{0})}$,
    $\mathcal{F}\paren{\mathbf{d}', T}$ and $\mathcal{G}\paren{\mathbf{d}', T}$ are both irregular.
\end{proof}

\section*{Acknowledgements}
The author would like to thank Professor Hiroki Sumi for his invaluable guidance and numerous helpful suggestions
throughout the preparation of this paper.
He has provided the author with continuous encouragement and instruction
over the past six years, from the fundamentals of mathematics to advanced topics.
The author is also grateful to Hiromichi Ohno for many helpful discussions.
Furthermore, the author would like to thank Masashi Kisaka, Kanji Inui, Takayuki Watanabe, and Yuto Nakajima for valuable comments and for answering questions.


\begin{thebibliography}{99}
    \bibitem{Falconer}
        K. J. Falconer,
        \textit{Fractal geometry: Mathematical foundations and applications, third edition},
        Wiley, 2014.
    \bibitem{Hutchinson1981}
        J. E. Hutchinson,
        \textit{Fractals and self-similarity}.
        Indiana Univ. Math. J. 30 (1981), 713–747.
    \bibitem{Mauldin1995}
        R. D. Mauldin,
        \textit{Infinite iterated function systems: Theory and applications}.
        Fractal Geometry and Stochastics, Progress in Probability 37, Birkh{\"a}user, Basel, 1995, pp. 91–110.
    \bibitem{MauldinUrbanski1996}
        R. D. Mauldin and M. Urba\'{n}ski,
        \textit{Dimensions and measures in infinite iterated function systems}.
        Proceedings of the London Mathematical Society 73 (1996), no. 1, 105–154.
    \bibitem{MauldinUrbanski1999}
        R. D. Mauldin and M. Urba\'{n}ski,
        \textit{Conformal iterated function systems with applications to the geometry of continued fractions}.
        Transactions of the American Mathematical Society 351 (1999), no. 12, 4995–5025.
    \bibitem{MauldinWilliams1986}
        R. D. Mauldin and S. C. Williams,
        \textit{Random recursive constructions: Asymptotic geometric and topological properties}.
        Transactions of the American Mathematical Society 295 (1986), no. 1, 325–346.
    \bibitem{Pommerenke1992}
        C. Pommerenke,
        \textit{Boundary Behaviour of Conformal Maps},
        Grundlehren der mathematischen Wissenschaften,
        Springer-Verlag, Berlin, Heidelberg, 1992.
    \bibitem{Sugita-en}
        H. Sugita,
        \textit{Dimension of limit sets of IFSs of complex continued fractions (in Japanese)}.
        Master's thesis, Osaka University, 2014, under supervision of H. Sumi.
    \bibitem{Takemoto-en}
        M. Takemoto,
        \textit{Properties of the family of CIFSs of generalized complex continued fractions (in Japanese)}.
        Master's thesis, Osaka University, 2015, under supervision of H. Sumi.
    \bibitem{yamaguchi-hata-kigami-en}
        M. Yamaguti, M. Hata, and J. Kigami,
        \textit{Mathematics of Fractals},
        Translations of Mathematical Monographs. American Mathematical Society, 1997.
        Translated into English by K. Hudson.
    \bibitem{Zagier-Katayama-1990-en}
        D. B. Zagier,
        \textit{Introduction to Number Theory: Zeta Functions and Quadratic Fields (in Japanese)}.
        Iwanami Shoten, Tokyo, 1990.
        Translated into Japanese by Koji Katayama. Original title: Zetafunktionen und quadratische Körper.
\end{thebibliography}
\end{document}